\documentclass[a4paper]{scrartcl}

\usepackage[english]{babel}
\usepackage{graphicx}
\usepackage{mathtools}         

\usepackage{amsmath, amssymb, amsfonts, amsthm} 
\usepackage[colorlinks=true,urlcolor=blue,linkcolor=black,citecolor=blue]{hyperref}

\usepackage{url}
\usepackage{hyperref}
\usepackage[hyphenbreaks]{breakurl}

\usepackage[dvipsnames]{xcolor}
\usepackage{color}          
\usepackage{changes}        
\usepackage{pdfpages}
\usepackage{enumerate}
\usepackage{stmaryrd} 

\usepackage{subcaption}

\DeclareMathOperator{\R}{\mathbb{R}}

\DeclareMathOperator{\N}{\mathbb{N}}
\DeclareMathOperator{\E}{\mathcal{E}}

\newcommand{\curv}{\vec{\kappa}}

\renewcommand{\S}{\mathbb{S}}

\renewcommand{\H}{\mathbb{H}}
\newcommand{\Haus}{\mathcal{H}}
\newcommand{\W}{\mathcal{W}}

\newcommand{\Ll}{\mathcal{L}}

\newcommand{\ssubset}{\subset\joinrel\subset}

\newcommand{\diver}{\mathrm{div}}

\newcommand{\measurerestr}{%
  \,\raisebox{-.127ex}{\reflectbox{\rotatebox[origin=br]{-90}{$\lnot$}}}\,%
}

\newcommand*{\dd}{\mathop{}\!\mathrm{d}}

\def\nicefrac#1#2{%
    \raise.5ex\hbox{$#1$}%
    \kern-.15em/\kern-.05em%
    \lower.25ex\hbox{$#2$}}

\numberwithin{equation}{section}

\usepackage[capitalize]{cleveref}
\crefname{equation}{}{}
\usepackage{autonum}

\theoremstyle{plain}
\newtheorem{lemma}{Lemma}[section]
\newtheorem{theorem}[lemma]{Theorem}

\newtheorem{proposition}[lemma]{Proposition}

\theoremstyle{definition}
\newtheorem{definition}[lemma]{Definition}
\newtheorem{remark}[lemma]{Remark}

\title{Global existence for the Willmore flow with boundary via Simon's Li-Yau inequality}
\author{\large{Manuel Schlierf\thanks{Institute of Applied Analysis, Ulm University, Helmholtzstra\ss e 18, 89081 Ulm, Germany. \texttt{manuel.schlierf@uni-ulm.de}}}}
\date{February 16, 2024}

\begin{document}
\maketitle
\begin{abstract}
\noindent\textbf{Abstract:} It is well-known that the Willmore flow of closed spherical immersions exists globally in time and converges if the initial datum has Willmore energy below $8\pi$ --- exactly the Li-Yau energy threshold below which all closed immersions are embedded. Extending the Li-Yau inequality for closed surfaces via Simon's monotonicity formula also for surfaces with boundary, given Dirichlet boundary conditions, one obtains an energy threshold $C_{\mathrm{LY}}$ below which surfaces with this boundary are embedded. By a slight modification, one obtains a threshold $C_{\mathrm{LY}}^{\mathrm{rot}}$ below which surfaces of revolution satisfying the boundary data have no self-intersections on the rotation axis.
 
With a new argument, using this modified Li-Yau inequality and tools from geometric measure theory, we show that the Willmore flow with Dirichlet boundary data starting in cylindrical surfaces of revolution exists globally in time if the energy of the initial datum is below $C_{\mathrm{LY}}^{\mathrm{rot}}$. Moreover, given Dirichlet boundary data, we also obtain the existence of a Willmore minimizer in the class of cylindrical surfaces of revolution if the corresponding infimum lies below $C_{\mathrm{LY}}^{\mathrm{rot}}$ which improves previous results for the stationary problem.
\end{abstract}

\bigskip
\noindent \textbf{Keywords:} Willmore flow, Willmore surfaces of revolution, Li-Yau inequality, monotonicity formula, geometric measure theory.

\noindent \textbf{MSC(2020)}: 53E40, 49Q10 (primary), 35B40, 35K41, 35J35 (secondary).

\section{Introduction}

With applications in modeling red blood cells in biology, cf. \cite{canham1970,helfrich1973}, or in general relativity, cf. \cite{hawking1968}, the Willmore functional is not only of interest to geometers but also in the field of applied mathematics. For an oriented, compact surface $\Sigma$ with or without boundary and an immersion $f\colon\Sigma\to\R^3$, one defines its \emph{Willmore energy}
\begin{equation}\label{eq:def-will-en}
    \W(f)=\int_{\Sigma} H^2\dd\mu_f
\end{equation}
where $\mu_f$ is the measure on $\Sigma$ induced by the pull-back metric $g_f=f^*\langle\cdot,\cdot\rangle$ of the Euclidean scalar product on $\R^3$. Further, writing $N$ for the smooth unit normal induced by the orientation of $\Sigma$, if $A_{ij}=\langle \partial^2_{ij}f,N\rangle$ are the components of the second fundamental form in local coordinates, then $H=\frac12 g^{ij}A_{ij}$ is the mean curvature of $f$. As usual, $(g^{ij})$ is the inverse of $(g_{ij})=(\langle\partial_if,\partial_jf\rangle)$.

In this article, we are concerned with the associated Dirichlet boundary value problem. To this end, suppose that $\partial\Sigma\neq\emptyset$ and denote by $\nu_f\colon\partial\Sigma\to\S^2\subseteq\R^3$ the unit outward-pointing co-normal field, cf. \cite[Proposition 2.17]{lee2018}. Then, given an immersion $b\colon\partial\Sigma\to\R^3$ of the boundary and $\eta\colon\partial\Sigma\to\S^2$ smooth with $\eta(x)\bot db(T_x\partial\Sigma)$ for all $x\in\partial\Sigma$, we are interested in the Dirichlet problem 
\begin{equation}\label{eq:stat-dir-prob}
    \begin{cases}
        \Delta H + |A^0|^2H = 0&\text{in $\Sigma$}\\
        f = b&\text{in $\partial\Sigma$}\\
        \nu_f=\eta&\text{in $\partial\Sigma$}
    \end{cases}
\end{equation}
where $\Delta$ is the Laplace-Beltrami operator with respect to $g_f$ and the trace-free second fundamental form $A^0$ is given by $A^0_{ij}=A_{ij}-g_{ij}H$ in local coordinates. Note that, in the literature, the boundary conditions in \cref{eq:stat-dir-prob} are also sometimes referred to as \emph{clamped} boundary conditions. Solutions to \cref{eq:stat-dir-prob} are exactly the stationary points of \cref{eq:def-will-en} and called \emph{Willmore surfaces}. 

Observe that \cref{eq:stat-dir-prob} is a fourth-order quasi-linear equation in $f$ which poses many analytical challenges, especially those due to the invariance of the Willmore functional with respect to conformal mappings in $\R^3$ which are smooth on $f(\Sigma)$.

In the literature, the partial differential equation is analyzed with various methods. Particularly, one can apply a variational approach as any Willmore minimizer satisfies \cref{eq:stat-dir-prob}. But also parabolic theory can be used by considering the associated gradient flow. With $b$ and $\eta$ as above, given an immersion $f_0\colon\Sigma\to\R^3$ with
\begin{equation}
    f_0(y)=b(y)\quad\text{and}\quad\nu_{f_0}(y)=\eta(y)\quad\text{ for all $y\in\partial\Sigma$},
\end{equation}
a family of immersions $f\colon[0,T)\times\Sigma\to\R^3$ satisfying
\begin{equation}\label{eq:dir-will-flow}
    \begin{cases}
        \partial_t f = - \bigl( \Delta H + |A^0|^2H \bigr)N&\text{in $[0,T)\times\Sigma$}\\
        f(0)=f_0&\text{in $\Sigma$}\\
        f(t)|_{\partial\Sigma}=b&\text{for all $t\in[0,T)$}\\
        \nu_{f(t)}=\eta&\text{for all $t\in[0,T)$}
    \end{cases}
\end{equation}
is called \emph{Willmore flow} with Dirichlet boundary conditions.

\subsection{State of the art}

The boundary problem \cref{eq:stat-dir-prob} is studied in \cite{schaetzle2010} for boundaries given by embedded, oriented, closed one-manifolds in $\R^n$. Following the variational approach to \cref{eq:stat-dir-prob}, a result in the class of varifolds is obtained in \cite[Theorem 4.1]{novagapozzetta2020}. The authors prescribe boundary data where $b$ parametrizes embedded, smooth curves. Then, so long as the associated infimum in the class of varifolds satisfying the given boundary data lies below $4\pi$, the minimum is attained. However, nothing is known in terms of further regularity results. Note that the notion of ``varifold'' is a rather weak concept and particularly does not encode a topology while the topology in \cref{eq:stat-dir-prob} is fixed by $\Sigma$. So the variational approach can fail in the class of immersions due to changes in the topology upon passing to (weak) limits while there are no issues in the realm of varifolds.

More concrete results have only been obtained for surfaces of revolution of cylindrical type. Write $C=[0,1]\times\S^1$ for the cylinder and consider an immersion $u\colon[0,1]\to\H^2$ where $\H^2=\R\times(0,\infty)$ denotes the hyperbolic plane. We denote by $f_u\colon C\to\R^3$ the associated surface of revolution given by
\begin{equation}
    f_u(x_1,x_2) = (u^{(1)}(x_1),u^{(2)}(x_1)\cos(x_2),u^{(2)}(x_1)\sin(x_2))^t.
\end{equation}
Uniqueness of \cref{eq:stat-dir-prob} is a delicate question and counterexamples in the class of cylindrical Willmore surfaces of revolution are given in \cite{eichmann2016}. Existence results for minimizers of the Willmore functional with certain Dirichlet boundary data are obtained in \cite{dallacquadeckelnickgrunau2008,dallacquafroehlichgrunauschieweck2011} in the class of surfaces of revolution arising from graphs, i.e. $u(x)=(x,h(x))^t$ for some positive function $h$ on $[0,1]$. This is achieved by cleverly choosing and analyzing minimizing sequences.

Moving from graphs to open curves $u\colon[0,1]\to\H^2$, in \cite[Theorem 1.1]{eichmanngrunau2019}, Eichmann-Grunau prove the existence of minimizers of $\W$ among cylindrical surfaces of revolution with Dirichlet boundary data given by two parallel circles of not necessarily equal radius and ``horizontal'' outer co-normal field if the corresponding infimum lies below $4\pi$. In this article, the result by Eichmann-Grunau is generalized to all possible Dirichlet data for cylindrical surfaces of revolution under an improved energy threshold.

Approaching \cref{eq:stat-dir-prob} with parabolic theory using the Willmore flow, not much is known for the Dirichlet problem. In \cite{kuwertschaetzle2002,kuwertschaetzle2001,kuwertschaetzle2004}, Kuwert-Schätzle famously show global existence and convergence to the round sphere for spherical immersions, i.e. $\Sigma=\S^2$, if the energy threshold $\W(f_0)<8\pi$ is satisfied. As these results rely on the classification of Willmore spheres by Bryant, cf. \cite{bryant1984}, obtaining similar results for other topologies also requires some understanding of the critical points, i.e. the stationary solutions. In \cite{dallacquamullerschatzlespener2020}, the authors use the blow-up construction of Kuwert-Schätzle and the classification of hyperbolic elastica due to \cite{langersinger1984} to obtain analogous results concerning the global existence and convergence of the Willmore flow for tori of revolution. 

Finally, the first study of the Willmore flow with Dirichlet boundary conditions is undertaken in \cite{schlierf2024} --- without using blow-up arguments but solely relying on one-dimensional energy estimates obtained by parabolic interpolation techniques. For initial data in the class of surfaces of revolution, \cite{schlierf2024} yields global existence and convergence of the Willmore flow if the initial datum satisfies an energy constraint which depends on the boundary conditions. This energy threshold is shown to be sharp for some critical boundary data.

\subsection{Main results}

Since we are working also in the class of cylindrical surfaces of revolution in this article, the Dirichlet boundary conditions can be described as follows. Let $p_0,p_1\in\H^2$, $\tau_0,\tau_1\in\S^1$ and write $p=(p_0,p_1)$ and $\tau=(\tau_0,\tau_1)$ for short. Define a first energy threshold by
\begin{equation}
    T(\tau)=4\pi - 2\pi \tau_y^{(2)}\Big|_{y=0}^1
\end{equation} 
where $\tau_y=(\tau_y^{(1)},\tau_y^{(2)})^t$. Moreover, writing
\begin{align}
    b_p(y,x_2)&=(p_y^{(1)},p_y^{(2)}\cos(x_2),p_y^{(2)}\sin(x_2))^t\quad\text{for $(y,x_2)\in\partial C=\{0,1\}\times\S^1$ and } \\\eta_\tau(y,x_2)&=(-1)^{1+y}\cdot(\tau_y^{(1)},\tau_y^{(2)}\cos(x_2),\tau_y^{(2)}\sin(x_2))^t\quad\text{ for $(y,x_2)\in\partial C$},\label{eq:def-bp-etatau}
\end{align}
define
\begin{align}
    C_{\mathrm{LY}}(p,\tau) &\vcentcolon= 8\pi - 2\sup_{z\in\R^3} \int_{\partial C} \frac{\langle\eta_{\tau}(x),b_p(x)-z\rangle}{|b_p(x)-z|^2}\dd\mu_{b_p^*\langle\cdot,\cdot\rangle}(x)\quad\text{and}\\
    C_{\mathrm{LY}}^{\mathrm{rot}}(p,\tau) &\vcentcolon= 8\pi - 2\sup_{z\in\R\times\{0\}\times\{0\}} \int_{\partial C} \frac{\langle\eta_{\tau}(x),b_p(x)-z\rangle}{|b_p(x)-z|^2}\dd\mu_{b_p^*\langle\cdot,\cdot\rangle}(x)\label{eq:def-cly-p-tau}
\end{align}
where $\mu_{b_p^*\langle\cdot,\cdot\rangle}$ is the measure on $\partial C$ induced by the pull-back of the Euclidean metric by $b_p$.  We argue later on that these thresholds are well-defined and actually satisfy the non-trivial relation $C_{\mathrm{LY}}(p,\tau)\geq T(\tau)$. Note that $C_{\mathrm{LY}}^{\mathrm{rot}}(p,\tau)\geq C_{\mathrm{LY}}(p,\tau)$ is trivial, using their definitions in \cref{eq:def-cly-p-tau}. Further, in \cref{app:est-cly-rot}, we deduce the explicit formula
\begin{equation}\label{eq:cly-rot-explicit}
    C_{\mathrm{LY}}^{\mathrm{rot}}(p,\tau) = 4\pi\Bigl[2-\sup_{h\in\R}\Bigl( p_y^{(2)}\frac{\tau_y^{(1)}(p_y^{(1)}-h)+\tau_y^{(2)}p_y^{(2)}}{(p_y^{(1)}-h)^2+(p_y^{(2)})^2}\Big|_{y=0}^1 \Bigr) \Bigr].
\end{equation}

\begin{remark}[Interpretation of the energy thresholds]
    Note that $T(\tau)$ is precisely the energy threshold below which the Willmore flow is understood in \cite{schlierf2024}. Moreover, if $f\colon C\to\R^3$ is an immersion with $f|_{\partial C}=b_p$ and $\nu_f=\eta_{\tau}$, then $\W(f)<C_{\mathrm{LY}}(p,\tau)$ yields that $f$ is an embedding if $b_p$ is an embedding, cf. \cref{prop:Li-Yau} --- this is an application of what we refer to as \emph{Simon's Li-Yau inequality} in the title, referencing Simon's monotonicity formula. Also cf. \cref{app:est-cly}. Moreover, if $\W(f)<C_{\mathrm{LY}}^{\mathrm{rot}}(p,\tau)$, then \cref{prop:Li-Yau} still yields that $f$ has no self-intersections on the rotation axis, i.e. on the line $\R\times\{0\}\times\{0\}$.
\end{remark}

In order to state the main results note that, if $u\colon[0,1]\to\H^2$ is an immersion, then
\begin{equation}\label{eq:bdry-cdts}
    u(y)=p_y\quad\text{and}\quad\frac{\partial_xu(y)}{|\partial_xu(y)|}=\tau_y\quad\text{for $y=0,1$}
\end{equation}
if and only if the associated surface of revolution $f_u$ satisfies
\begin{equation}
    f_u|_{\partial C} = b_p\quad\text{and}\quad \nu_{f_u} = \eta_{\tau}.
\end{equation}
The first result concerns the variational approach to \cref{eq:stat-dir-prob}.

\begin{theorem}\label{thm:main-calvar}
    Let $p_0,p_1\in\H^2$ and $\tau_y\in \S^1$ for $y=0,1$ and consider
    \begin{equation}\label{eq:min-prob}
        M_{p,\tau} \vcentcolon= \inf\{\W(f_u)\mid u\in W^{2,2}([0,1],\H^2)\text{ is an immersion satisfying \cref{eq:bdry-cdts}}\}.
    \end{equation}    
    For $M_{p,\tau}<C_{\mathrm{LY}}^{\mathrm{rot}}(p,\tau)$, the infimum in \cref{eq:min-prob} is attained by a smooth cylindrical Willmore-surface of revolution.
\end{theorem}

For boundary data $p,\tau$ as discussed in \cite[Theorem 1.1]{eichmanngrunau2019}, i.e. $p_0=(-1,\alpha_-)^t$, $p_1=(1,\alpha_+)^t$ and $\tau_y=(1,0)^t$ for some $\alpha_-,\alpha_+>0$, one finds $C_{\mathrm{LY}}^{\mathrm{rot}}(p,\tau)\geq T(\tau) = 4\pi$ --- thus recovering the result of Eichmann-Grunau from \cref{thm:main-calvar}. Moreover, using the explicit computations of $C_{\mathrm{LY}}^{\mathrm{rot}}$ in \cref{app:est-cly-rot} in this setting, one finds that \cref{thm:main-calvar} is actually an improvement.

Note that, independently of the work for this article, Eichmann-Schätzle solve the existence problem \cref{eq:min-prob} in the recent preprint \cite{eichmannschaetzle2024} for all Dirichlet boundary data apart from a degenerate case in this rotationally symmetric setting.

It is well-known that the Willmore flow starting in a surface of revolution parametrizes a surface of revolution at every time in its maximal existence interval and used for instance in \cite{dallacquamullerschatzlespener2020,schlierf2024}. Using \cite[Theorems 3.14 and 4.5]{schlierf2024}, if $f_u\colon[0,T)\times C\to\R^3$ is a maximal solution of \cref{eq:dir-will-flow} and $\sup_{t\in[0,T)}\Ll_{\H^2}(u(t))<\infty$, then $T=\infty$ and $f_{u(t)}$ converges after reparametrization to a Willmore surface of revolution for $t\to\infty$. Here, $\Ll_{\H^2}$ denotes the length of a curve in the hyperbolic plane. In this article, we give a new argument for the boundedness of the hyperbolic lengths by combining Simon's Li-Yau inequality with methods from geometric measure theory. After proving that \cref{eq:en-thresh} is sufficient for the hyperbolic lengths of the profile curves to remain uniformly bounded, we obtain the following

\begin{theorem}\label{thm:main}
    Let $u_0\colon[0,1]\to\H^2$ be an immersion satisfying \cref{eq:bdry-cdts}
    and
    \begin{equation}\label{eq:en-thresh}
        \W(f_{u_0}) \leq C_{\mathrm{LY}}^{\mathrm{rot}}(p,\tau).
    \end{equation}
    If $f_u\colon[0,T)\times C\to\R^3$ is a maximal solution of \cref{eq:dir-will-flow} with $f_0=f_{u_0}$, $b=b_p$ and $\eta=\eta_{\tau}$, then $T=\infty$ and the solution $f_u$ converges for $t\to\infty$ up to reparametrization smoothly to a cylindrical Willmore surface of revolution, i.e. a solution of \cref{eq:stat-dir-prob}.
\end{theorem}

In \cref{prop:battle-of-thresholds}, we prove that $C_{\mathrm{LY}}(p,\tau)\geq T(\tau)$. It is clear that one can have a strict inequality: For instance, if $p_0=p_1$ and $\tau_0=\tau_1$, then $C_{\mathrm{LY}}^{\mathrm{rot}}(p,\tau)=C_{\mathrm{LY}}(p,\tau)=8\pi$ while $T(\tau)=4\pi$. Therefore, \cref{thm:main} strictly improves \cite[Theorem 1.1]{schlierf2024}.

Again, independently of the work for this article, adapting their arguments from \cite{eichmannschaetzle2024}, in the recent preprint \cite{eichmann2024}, Eichmann also gives an improved energy threshold for the initial datum in \cref{eq:en-thresh}. While the argument in a way also involves excluding certain self-intersections on the rotation axis, Eichmann does not rely on Simon's Li-Yau inequality --- which as you may note is not specific to the rotational setting but could also be applied more generally. Instead, Eichmann makes use of the classification of rotationally symmetric critical points of the Willmore functional in \cite{langersinger1984} to construct a new energy bound which is specific to the rotationally symmetric setting, using an explicit geometric construction involving cut-outs of spheres and inverted catenoids modulo Möbius transformations. Notice that this corresponds to the findings on ``singular limits'' of the hyperbolic elastic flow in \cite{schlierf2023} which turn out to be exactly of this type. Both flows, \cref{eq:dir-will-flow} and the hyperbolic elastic flow are actually closely related, cf. \cite[Theorem~4.1]{dallacquaspener2018}. In \cref{app:est-cly-rot}, using Simon's monotonicity formula, we prove that both energy thresholds, i.e. \cref{eq:en-thresh} and the one in \cite[Theorem 1.1]{eichmann2024} are equivalent --- a non-trivial observation that provides a nice geometric interpretation for $C_{\mathrm{LY}}^{\mathrm{rot}}$.

\subsection{Strategy of proof and outline}

\cref{sec:not} contains a brief summary of geometric preliminaries and notation. Especially, we review the fundamentals of varifolds and Simon's Li-Yau inequality.

As aforementioned, the strategy of proof of \cref{thm:main-calvar,thm:main} is based on showing boundedness of the hyperbolic length of profile curves. As a first step, we establish that the energy bound $8\pi$ is sufficient for the diameters to be bounded in \cref{prop:bd-euc-len} --- this uses the preliminary observations of \cref{sec:prelim}. 

Then, one argues that an explosion of the hyperbolic length can only happen if the profile curves asymptotically touch the rotation axis. In this case, passing to a weak limit, one obtains a limiting varifold with the same Dirichlet boundary as the sequence of surfaces of revolution but with a point on the rotation axis which has density 2. Now the energy threshold and Simon's Li-Yau inequality yield a contradiction and thus boundedness of the hyperbolic lengths in the first place. This is contained in \cref{prop:prop2-bd-hyp-len}. The remainder of \cref{sec:pf-main} contains the proofs of \cref{thm:main-calvar,thm:main}.

Finally, in \cref{app:est-cly}, the relation $C_{\mathrm{LY}}(p,\tau)\geq T(\tau)$ is shown using Simon's monotonicity formula and some geometric constructions involving cut-outs of spheres and in \cref{app:est-cly-rot}, we show the explicit formula \cref{eq:cly-rot-explicit} for $C_{\mathrm{LY}}^{\mathrm{rot}}(p,\tau)$.


\section{Geometric preliminaries and notation}\label{sec:not}

In this article, we consider surfaces of revolution of cylindrical type which are immersed in $\R^3$. Write $C=[0,1]\times\S^1$ for the cylinder and consider an immersion $u\colon[0,1]\to\H^2$ where $\H^2=\R\times(0,\infty)$ denotes the hyperbolic plane. We denote by $f_u\colon C\to\R^3$ the associated surface of revolution given by
\begin{equation}
    f_u(x_1,x_2) = (u^{(1)}(x_1),u^{(2)}(x_1)\cos(x_2),u^{(2)}(x_1)\sin(x_2))^t.
\end{equation}
Fix the orientation $\{\frac{\partial}{\partial x_1},\frac{\partial}{\partial x_2}\}$ on the cylinder $C$ and equip it with the pull-back metric $g_{f_u}=(f_u)^*\langle\cdot,\cdot\rangle$ of the Euclidean scalar product on $\R^3$, that is, $g_{ij}=\langle\partial_if_u,\partial_jf_u\rangle$. Then the outer normal vector $N$ of $f_u$ induced by the orientation of $C$ is given by
\begin{equation}
    N=\frac{\partial_{x_1}f_u\times\partial_{x_2}f_u}{|\partial_{x_1}f_u\times\partial_{x_2}f_u|}
\end{equation} 
and the (scalar) second fundamental form by $A_{ij}=\langle \partial^2_{ij}f,N\rangle$. The mean curvature can be computed as $H=\frac12g^{ij}A_{ij}$ and the tensor $A^0$ with components $A^0_{ij} = A_{ij} - Hg_{ij}$ is called trace-free second fundamental form. The \emph{Willmore energy} is
\begin{equation}
    \W(f_u) = \int_{C} H^2\dd\mu_{f_u}
\end{equation}
where $\mu_{f_u}$ is the measure on $C$ induced by $g_{f_u}$. For a family of immersions $f\colon[0,T)\times C\to\R^3$, write $\nabla\W(f)=(\Delta H + |A^0|^2H)N$ where $\Delta$ denotes the Laplace-Beltrami operator on $C$ with respect to $g_{f}$.

These cylindrical surfaces of revolution are completely determined by their profile curves $u$ in the hyperbolic plane $\H^2$. Using the half-plane model, $\H^2$ is equipped with the metric $g_p=1/(p^{(2)})^2\langle\cdot,\cdot\rangle$ where $p=(p^{(1)},p^{(2)})^t\in\H^2$ and $\langle\cdot,\cdot\rangle$ is the Euclidean scalar product in $\R^2$. Denote covariant differentiation on $\H^2$ by $\nabla$. For an immersion $u\colon[0,1]\to\H^2$, write $\dd s=|\partial_xu|_g\dd x$ for the arc-length element, $\partial_s=1/|\partial_xu|_g\partial_x$ and $\curv=\nabla_{\partial_su}\partial_su$ for its curvature. The \emph{elastic energy} of $u$ is defined by
\begin{equation}
    \E(u)=\int_0^1|\curv|_g^2\dd s.
\end{equation} 
In the following, we consider the following Dirichlet boundary conditions for cylindrical surfaces of revolution. Let $p_1,p_1\in\H^2$ and $\tau_0,\tau_1\in \S^1\subseteq \R^2$. If $u\colon[0,T)\times[0,1]\to\H^2$ is a family of immersions satisfying \cref{eq:bdry-cdts}, one finds as in \cite[pp. 9 -- 12]{mingione2012}
\begin{equation}\label{eq:l2-grad-will}
    \frac{\dd}{\dd t}\W(f_u) = \int_{C} \langle \nabla\W(f_u),\partial_t f\rangle\dd\mu_{f_u}.
\end{equation}
Furthermore, writing $\nabla\E(u)=2(\nabla_s^{\bot})^2\curv+|\curv|_g^2\curv-2\curv$ where $v^{\bot}=v-\langle v,\partial_su(x)\rangle_g\partial_su(x)$ for $v\in T_{u(x)}\H^2$, one finds for any family of smooth immersions $u\colon[0,T)\times[0,1]\to\H^2$ satisfying \cref{eq:bdry-cdts}, as in \cite[Remark 2.5]{dallacquaspener2017},
\begin{equation}
    \frac{\dd}{\dd t}\E(u) = \int_0^1\langle \nabla\E(u),\partial_t u\rangle_g\dd s.
\end{equation}
Consider an immersion $u_0\colon[0,1]\to\H^2$ satisfying \cref{eq:bdry-cdts}.
If suitable further compatibility conditions are satisfied, there exists a maximal solution $u\colon[0,T)\times[0,1]\to\H^2$ to 
\begin{equation}\label{eq:wf-eq}
\begin{cases}
	\partial_t u = -\frac{1}{4(u^{(2)})^4}\nabla\E(u)&\text{in $[0,T)\times[0,1]$}\\
    u(0)=u_0&\text{in $[0,1]$}\\
    u(t,y)=p_y&\text{for $t\in[0,T)$ and $y\in\{0,1\}$}\\
    \frac{\partial_xu(t,y)}{|\partial_xu(t,y)|}=\tau_y&\text{for $t\in[0,T)$ and $y\in\{0,1\}$},
\end{cases}
\end{equation}
cf. \cite[Appendix C]{schlierf2024} for some references for short-time existence results and a proof of uniqueness. Then computations in \cite[Theorem 4.1]{dallacquaspener2018} show that the associated surfaces of revolution $f_u$ in $\R^3$ satisfy the classical Willmore flow evolution \cref{eq:dir-will-flow} with $\eta=\eta_{\tau}$ and $b=b_p$. Therefore, a solution $u$ to \cref{eq:wf-eq} is also called \emph{Willmore flow} with Dirichlet boundary conditions.

The following fundamental relation goes back to \cite{bryantgriffiths1986}. In its cited form, it can be found in \cite[(2.11)]{eichmanngrunau2019}. Writing $|\cdot|$ for the Euclidean norm in $\R^2$, for an immersion $u\colon I\to\H^2$ on a compact interval $I\subseteq\R$, one has
\begin{equation}\label{eq:br-gr}
    \frac{2}{\pi}\W(f_u) = \E(u) - 4\frac{\partial_xu^{(2)}}{|\partial_xu|}\Big|_{\partial I}.
\end{equation}
\begin{remark}
    Consider a compact interval $I\subseteq\R$ and an immersion $u\in W^{2,2}(I)$ into $\H^2$. By \cite{meyersserrin1964}, there exists a sequence $(u_n)_{n\in\N}\subseteq C^{\infty}(I^{\circ})\cap W^{2,2}(I)$ with $u_n\to u$ in $W^{2,2}(I)$ and thus also in $C^1(I)$. Using \cref{lem:elen-cont}, $\E(u_{n})\to \E(u)$ and, applying \cref{eq:br-gr} to each $u_{n}$, we obtain
    \begin{equation}\label{eq:app-br-gr}
        \E(u)\geq 4\frac{\partial_xu^{(2)}}{|\partial_xu|}\Big|_{\partial I}.
    \end{equation}
\end{remark}

\subsection{Some properties of varifolds and immersed manifolds}

Detailed references and computations for the following definitions and results can be found in \cite[Section 1.1 and Example 2.4]{scharrer2022}. First, we review the definition and basic properties of varifolds. For integers $1\leq m<n$, denote by $G(n,m)$ the set of all $m$-dimensional subspaces of $\R^n$ which has a natural structure as a smooth Euclidean submanifold. An $m$-dimensional varifold $V$ in $\R^n$ is then a Radon-measure on $G_m(\R^n)\vcentcolon= \R^n\times G(n,m)$. Its associated \emph{weight measure} $\mu_V$ on $\R^n$ is the push-forward of $V$ under the projection $G_m(\R^n)\to\R^n$, i.e.
\begin{equation}
    \mu_V(A)=V(\{(p,T)\in G_m(\R^n):p\in A\})\quad\text{for all $A\subseteq \R^n$}.
\end{equation}
If $X\in C_c^1(\R^n,\R^n)$ and $T\in G(n,m)$, one writes $\diver_T X=\sum_{i=1}^m \langle D_{e_i}X,e_i\rangle_{\R^n}$ for any orthonormal basis $e_1,\dots,e_m$ of $T$. The first variation of $V$ is then defined as the functional $\delta V\colon C_c^1(\R^n)\to \R$ with 
\begin{equation}
    \delta V(X) = \int_{G_m(\R^n)} \diver_T(X) (x)\dd V(x,T)
\end{equation}
and $\|\delta V\|$ with $\|\delta V\|(U)=\sup\{\delta V(X):X\in C_c^1(\R^n),\ \mathrm{supp}(X)\subseteq U,\ |X|\leq 1\}$ for $U\subseteq\R^n$ open is called its \emph{total variation}. Finally, one says that $\vec{H}$ is the generalized mean curvature of $V$ if it is $\mu_V$-measurable, $\|\delta V\|$ is a Radon measure over $\R^n$ and if there exists a $\|\delta V\|$-measurable map $\eta_V$ with $|\eta_V(x)|\leq 1$ for $\|\delta V\|$-a.e. $x$ and 
\begin{equation}
    \delta V(X) = - m \int_{\R^n} \langle \vec{H},X\rangle_{\R^n}\dd\mu_V + \int_{\R^n} \langle X,\eta_V\rangle_{\R^n} \dd\sigma_V
\end{equation}
where $\sigma_V=\|\delta V\| - \|\delta V\|_{\mu_V}$ where $\|\delta V\|_{\mu_V}$ is the absolutely continuous part of $\|\delta V\|$ with respect to $\mu_V$. One defines the Willmore energy of $V$ by 
\begin{equation}
    \W(V) = \int_{\R^n} |\vec{H}|^2\dd\mu_V.
\end{equation}
Viewing varifolds as Radon measures provides a natural concept of weak varifold convergence. Namely, we say that a sequence of $m$-dimensional varifolds $(V_j)_{j\in\N}$ weakly converges to an $m$-dimensional varifold $V$ in $\R^n$ and write $V_j\rightharpoonup V$ if 
\begin{equation}
    \lim_{j\to\infty} \int_{G_m(\R^n)}\varphi(x,T)\dd V_j(x,T) = \int_{G_m(\R^n)} \varphi(x,T)\dd V(x,T)
\end{equation}
for all $\varphi\in C_c^0(G_m(\R^n))$. Note that both the total variation and the Willmore energy are lower semi-continuous with respect to weak varifold convergence.

There is a natural way to associate an $m$-dimensional varifold $V_f$ to an immersion $f\colon\Sigma\to\R^n$ of an $m$-dimensional Riemannian manifold $\Sigma$ into $\R^n$. Indeed, for any $\varphi\in C_c^0(G_m(\R^n))$, define 
\begin{equation}
    V_f(\varphi) = \int_{\R^n} \sum_{p\in f^{-1}(\{x\})} \varphi(x,df_p(T_p\Sigma))\dd\Haus^m(x)
\end{equation}
where $\Haus^m$ simply denotes the $m$-dimensional Hausdorff measure on $\R^n$. One then has that $\mu_{V_f}=f_{\#}\mu_f$ and $\theta^m(\mu_{V_f},x)=\Haus^0(f^{-1}(\{x\}))$ for all $x\in f(\Sigma\setminus\partial\Sigma)$ where $\theta^m(\mu,x)$ denotes the $m$-dimensional density of the measure $\mu$ at $x$, if it exists. Denoting by $\vec{H}_f$ the classical mean curvature field of $f$, i.e. $\frac{1}{m}$ times the trace of the vector-valued second fundamental form, define the $\mu_{V_f}$-measurable map $\vec{H}\colon\R^n\to\R^n$ by
\begin{equation}
    \vec{H}(x)=\begin{cases}
        \frac{1}{\theta^m(\mu_{V_f},x)} \sum_{p\in f^{-1}(\{x\})} \vec{H}_f(p)&\text{if $\theta^m(\mu_{V_f},x)> 0$}\\
        0&\text{else}.
    \end{cases}
\end{equation} 
By \cite[Proposition 2.17]{lee2018}, if $\partial\Sigma\neq\emptyset$, there exists a unique outward-pointing unit co-normal field $\nu$ along $\partial\Sigma$. If $f$ is isometric, writing $\nu_f=df(\nu)\colon\partial\Sigma\to\S^{n-1}\subseteq\R^n$ for the co-normal of $f$ on the boundary, $V_f$ has generalized mean curvature $\vec{H}$ and 
\begin{equation}
    \delta V_f(X) = - m\int_{\R^n} \langle X,\vec{H}\rangle_{\R^n} \dd\mu_{V_f} + \int_{\partial\Sigma} \langle X\circ f, \nu_f \rangle_{\R^n}\dd\sigma_f
\end{equation}
where $\sigma_f$ is the induced measure on $\partial\Sigma$ by $(f|_{\partial\Sigma})^*\langle\cdot,\cdot\rangle_{\R^n}$.

We also define the notion of integral varifolds. An $m$-dimensional varifold $V$ in $\R^n$ is said to be \emph{integral} if there are countably many $C^1$-regular submanifolds $M_i$ and Borel subsets $B_i\subseteq M_i$ such that $V=\sum_{i=1}^{\infty} v_m(B_i)$ where 
\begin{equation}
    v_m(B_i)(\varphi) = \int_{B_i} \varphi(x,T_xM_i)\dd\Haus^m(x)
\end{equation}
for any $\varphi\in C_c^0(G_m(\R^n))$. One obtains that $\mu_V=\theta \Haus^m \measurerestr M$ for $\theta\in L^1_{\mathrm{loc}}(\Haus^m\measurerestr M)$ with values in $\N_0$ and a countably $m$-rectifiable set $M\subseteq\R^n$. Since immersions are locally embeddings, one obtains that any $V_f$ as defined above is integral.

An important property of such integral varifolds is the following compactness theorem due to Allard in \cite{allard1972}.

\begin{theorem}\label{thm:vari-cpt}
    Consider a sequence of $m$-dimensional integral varifolds $(V_j)_{j\in\N}$ in $\R^n$ with 
    \begin{equation}
        \sup_{j\in\N} \|\delta V_j\|(W) + \mu_{V_j}(W) < \infty\quad\text{for all $W\ssubset\R^n$}.
    \end{equation}
    Then there exists an $m$-dimensional integral varifold $V$ with $V_j\rightharpoonup V$, after passing to a subsequence.
\end{theorem}

\subsection{Simon's Li-Yau inequality}

In the setting of integral varifolds with finite Willmore energy, using Simon's monotonicity formula for the Willmore energy, one has the following Li-Yau inequality which is originally proved for closed surfaces in \cite{liyau1982}. Below we obtain a Li-Yau-type inequality for varifolds with not necessarily vanishing singular part of the total variation. The idea of employing a monotonicity formula to find new Li-Yau-type inequalities is also used in \cite{ruppscharrer2023} to obtain a Li-Yau inequality for the Helfrich functional for closed varifolds.

\begin{proposition}\label{prop:Li-Yau}
    Let $V$ be a $2$-dimensional integral varifold in $\R^3$ with generalized mean curvature $\vec{H}$ and suppose that $\W(V)<\infty$, $\mu_V(\R^3)<\infty$ and that $\mathrm{supp}(\sigma_V)$ is compact with $\sigma_V(\R^3)<\infty$. For all $z\in\R^3\setminus\mathrm{supp}(\sigma_V)$,
    \begin{equation}\label{eq:monotonicity}
        \theta^2(\mu_V,z)\pi \leq \frac{\W(V)}{4} + \frac{1}{2} \int_{\R^3}  \frac{\langle \eta_V(x),x-z \rangle}{|x-z|^2}\dd \sigma_V.
    \end{equation}
    Therefore, if $K\subseteq\R^3\setminus\mathrm{supp}(\sigma_V)$ is non-empty and 
    \begin{equation}\label{eq:Li-Yau}
        \W(V) < 8\pi - 2\sup_{z\in K} \int_{\R^3} \frac{\langle \eta_V(x),x-z \rangle}{|x-z|^2}\dd \sigma_V,
    \end{equation}
    then $\theta^2(\mu_V,z)<2$ for all $z\in K$.
\end{proposition}
\begin{proof}
    This is a direct consequence of \cref{eq:mon-form} upon noting that $\lim_{t\searrow 0}A_z(t)=\theta^2(\mu_V,z)\pi$ and $\lim_{t\nearrow\infty}A_z(t)=\frac{\W(V)}{4} + \frac{1}{2} \int_{\R^3}  \frac{\langle \eta_V(x),x-z \rangle}{|x-z|^2}\dd \sigma_V$, cf. \cite[(B.8) and p. 555]{novagapozzetta2020}.
\end{proof}

Motivated by this Li-Yau inequality, we make the following definition.

\begin{definition}\label{def:def-cly}
    Let $\Sigma$ be a compact surface with boundary, $b\colon \partial\Sigma\to\R^3$ an immersion and $\eta\colon\partial\Sigma\to\S^2\subseteq\R^3$ with $\eta(x)\bot db(T_x\partial\Sigma)$ for all $x\in\partial\Sigma$ and $K\subseteq\R^3$. Then 
    \begin{equation}\label{eq:def-cly}
        C_{\mathrm{LY}}^K(b,\eta) \vcentcolon= 8\pi - 2\sup_{z\in K} \int_{\partial\Sigma} \frac{\langle\eta(x),b(x)-z\rangle}{|b(x)-z|^2}\dd\mu_{b^*\langle\cdot,\cdot\rangle}(x)
    \end{equation}
    where $\mu_{b^*\langle\cdot,\cdot\rangle}$ is the measure on $\partial\Sigma$ induced by the metric $b^*\langle\cdot,\cdot\rangle$.
\end{definition}
\begin{remark}\label{rem:p-tau-for-vari}
    If $u\colon[0,1]\to\H^2$ is an immersion with $p_y=u(y)$ and $\tau_y=\partial_xu(y)/|\partial_xu(y)|$ for $y=0,1$, then 
    \begin{equation}
        C_{\mathrm{LY}}^{\R\times\{0\}\times\{0\}}(f_u|_{\partial C},\nu_{f_u}) =C_{\mathrm{LY}}^{\R\times\{0\}\times\{0\}}(b_p,\eta_{\tau})= C_{\mathrm{LY}}^{\mathrm{rot}}(p,\tau),
    \end{equation}
    using \cref{eq:def-bp-etatau}. Particularly, if $V$ is a $2$-dimensional integral varifold in $\R^3$ with finite Willmore energy and area such that $\sigma_V=\sigma_{V_{f_u}}$ and $\eta_V=\eta_{V_{f_u}}$, then \cref{eq:Li-Yau} with $K=\R\times\{0\}\times\{0\}$ is equivalent to
    \begin{equation}
        \W(V) < C_{\mathrm{LY}}^{\mathrm{rot}}(p,\tau).
    \end{equation}
\end{remark}
\begin{remark}\label{rem:cly-vs-8pi}
    As $b\colon\partial\Sigma\to\R^3$ and $\eta\colon\partial\Sigma\to\S^{2}$ are smooth, one can argue that the integral in \cref{eq:def-cly} is finite for any $z\in K$, cf. \cite[Equation (4.2)]{novagapozzetta2020}. Since the integrand converges to $0$ for $|z|\to\infty$, if $K$ is unbounded, the supremum in \cref{eq:def-cly} is non-negative and therefore $C_{\mathrm{LY}}^K(b,\eta)\leq 8\pi$. Particularly, also $C_{\mathrm{LY}}^{\mathrm{rot}}(p,\tau)\leq 8\pi$.
    
    Returning to \cref{eq:def-cly-p-tau}, the estimate $C_{\mathrm{LY}}(p,\tau)\geq 0$ is less obvious but follows from \cref{prop:battle-of-thresholds}.
\end{remark}


\section{Fundamental estimates for open curves in $\H^2$}\label{sec:prelim}

\begin{proposition}\label{prop:eich-grun}
    Consider immersions $u_n\colon[0,1]\to\H^2$ and points $p_0,p_1\in\H^2$ such that $u_n(y)\in B_R(p_y)\subseteq\H^2$ for some $R>0$, for $y=0,1$. Then
    \begin{equation}\label{eq:prop-eich-grun}
        \limsup_{n\to\infty}\Ll_{\H^2}(u_n)=\infty \implies \limsup_{n\to\infty}\E(u_n)\geq 8.
    \end{equation}
    Particularly, 
    \begin{equation}\label{rem:eich-grun}
        \lim_{n\to\infty}\Ll_{\H^2}(u_n)=\infty \implies \liminf_{n\to\infty}\E(u_n)\geq 8.
    \end{equation} 
\end{proposition}
\begin{proof}
    The first implication is proved with the same arguments as \cite[Corollary 2.7]{schlierf2024}. For \cref{rem:eich-grun}, Suppose that $\liminf_{n\to\infty}\E(u_n)< 8$. Then there exists a subsequence $(u_{n_k})_{k\in\N}$ with $\limsup_{k\to\infty}\E(u_{n_k})<8$. However, $\limsup_{k\to\infty}\Ll_{\H^2}(u_{n_k})=\lim_{n\to\infty}\Ll_{\H^2}(u_{n})=\infty$ then contradicts \cref{eq:prop-eich-grun}.
\end{proof}

\begin{proposition}\label{prop:euc-len-diam}
    Consider an immersion $u\in W^{2,2}(I)$, $u\colon I\to\H^2$ on a compact interval $I\subseteq\R$ and suppose that $\E(u)\leq M$ as well as $u^{(2)}\leq R$. Then there exists $C=C(M)$ with
    \begin{equation}
        \Ll_{\R^2}(u)\leq C(M)\cdot R.
    \end{equation}
\end{proposition}
\begin{proof}
    W.l.o.g. suppose that $|\partial_xu|\equiv L=\Ll_{\R^2}(u)$ and $I=[0,1]$. Furthermore, after approximating $u$ in $W^{2,2}(I)$ by a sequence of $C^{\infty}$-immersions, we may w.l.o.g. suppose that $u$ is smooth. Then the computations in the proof of \cite[Lemma 2.5]{dallacquamullerschatzlespener2020} yield that 
    \begin{equation}
        2L - 2\frac{\partial_x u^{(2)}}{L}u^{(2)}\Big|_{\partial I} \leq C(\W(f_u)) (\mu_{f_u}(\Sigma))^{\frac12} \leq C(M)(\mu_{f_u}(\Sigma))^{\frac12}.
    \end{equation}
    Therefore, $L\leq C(M) (\mu_{f_u}(\Sigma))^{\frac12} + 2R$. Since furthermore
    \begin{equation}
        \mu_{f_u}(\Sigma) = 2\pi L \int_0^1 u^{(2)}\dd x \leq 2\pi L R,
    \end{equation}
    the claim follows.
\end{proof}

\subsection{On curves with infinite hyperbolic length}

As already noted in previous works, when studying the Willmore functional in the class of surfaces of revolution, encountering singular behavior is closely related to encountering profile curves with unbounded hyperbolic length, for instance cf. \cite{muellerspener2020,dallacquamullerschatzlespener2020,eichmanngrunau2019,schlierf2024}. Therefore, we investigate the behavior of such profile curves at ``hyperbolic infinity'' --- as it turns out, the tangent vectors become more and more vertical, c.f. \cref{eq:lem-vert-0}. The general strategy used in the next Lemma follows arguments in \cite[proof of Lemma 4.2]{eichmanngrunau2019}.

\begin{lemma}\label{lem:vert}
    Consider $u\colon[a,b)\to\H^2$ where $b\in\R\cup\{\infty\}$ such that $u\in W^{2,2}([a,r])$ for all $a<r<b$. Assume that $u$ is an immersion, i.e. $|\partial_xu|>0$. Suppose that $\Ll_{\H^2}(u)=\infty$, $\E(u)<\infty$. Then there exists a sequence $(x_n)_{n\in\N}\subseteq [a,b)$ with $x_n\to b$ such that
    \begin{equation}\label{eq:lem-vert-0}
        \frac{|\partial_xu^{(2)}(x_n)|}{|\partial_xu(x_n)|} \to 1\quad\text{for $n\to\infty$}.
    \end{equation} 
\end{lemma}
\begin{proof}
    Note that $u\in C^1([a,b))$. Suppose that no such sequence $(x_n)_{n\in\N}$ satisfying \cref{eq:lem-vert-0} exists. Then, there are $a'\in[a,b)$ and $\delta>0$ with $\frac{|\partial_xu^{(2)}(x)|}{|\partial_xu(x)|}\leq 1-2\delta$ for all $x\in [a',b)$. 
    
    Denote by $(\varphi_{\varepsilon})_{\varepsilon>0}$ the standard mollifiers on $\R$ and write $u_{\varepsilon} = \varphi_{\varepsilon} * u$. Then $u_{\varepsilon} \to u$ in $W^{2,2}(I)$ for any compact interval $I\subseteq (a,b)$. Particularly, $u_{\varepsilon}$ is an immersion for $\varepsilon$ sufficiently small and, for $I\subseteq (a,b)$ compact,
    \begin{equation}
        \E(u_{\varepsilon}|_I) \to \E(u|_I)\quad\text{and}\quad \Ll_{\H^2}(u_{\varepsilon}|_I)\to \Ll_{\H^2}(u|_I).
    \end{equation}
    Define $C_{\delta} = \frac{1-\delta}{\delta}$ and choose some $a'<r<b$ such that 
    \begin{equation}\label{eq:lem-vert-1}
        \Ll_{\H^2}(u|_{[a',r]}) > (1+C_{\delta}^2)(\E(u)+5)+1.
    \end{equation}
    Furthermore, there exists $\varepsilon'>0$ sufficiently small such that
    \begin{equation}\label{eq:lem-vert-2}
        |\E(u_{\varepsilon'}|_{[a',r]})-\E(u|_{[a',r]})|<1\text{, }|\Ll_{\H^2}(u_{\varepsilon'}|_{[a',r]})-\Ll_{\H^2}(u|_{[a',r]})|<1 \text{ and } 
    \end{equation}
    $\frac{|\partial_xu_{\varepsilon'}^{(2)}|}{|\partial_xu_{\varepsilon'}|}\leq 1-\delta$ on $[a',r]$ and such that $u_{\varepsilon'}$ maps into $\H^2$. Especially,
    \begin{equation}
        (1-(1-\delta)^2)({\partial_xu_{\varepsilon'}^{(2)}})^2 \leq (1-\delta)^2({\partial_xu_{\varepsilon'}^{(1)}})^2
    \end{equation} 
    so that $u_{\varepsilon'}^{(1)}\colon [a',r]\to u_{\varepsilon'}^{(1)}([a',r])=J$ is a diffeomorphism. Writing $g=u_{\varepsilon'}^{(2)}\circ ({u_{\varepsilon'}^{(1)}})^{-1}\colon J\to (0,\infty)$, one obtains $\widetilde{u_{\varepsilon'}}(y)=u_{\varepsilon'}\circ (u_{\varepsilon'}^{(1)})^{-1}(y)=(y,g(y))$ as a reparametrization of $u_{\varepsilon'}$ as a graph. Moreover, $\frac{|\partial_xu_{\varepsilon'}^{(2)}(x)|}{|\partial_xu_{\varepsilon'}(x)|}\leq 1-\delta$ yields that $|\partial_yg(y)|\leq C_{\delta}$ for all $y\in J$. Then, using \cref{eq:lem-vert-2}, and the explicit formula for the elastic energy of graphs in \cite[(2.6)]{schlierf2024},
    \begin{align}
        \E(u)+5&\geq\E(\widetilde{u_{\varepsilon'}}|_{J}) +4 \geq \int_J \frac{1}{g(y)\sqrt{1+(\partial_yg(y))^2}}\dd y \geq \frac{1}{1+C_{\delta}^2} \int_J \frac{\sqrt{1+(\partial_yg(y))^2}}{g(y)}\dd y \\
        &= \frac{1}{1+C_{\delta}^2} \Ll_{\H^2}(\widetilde{u_{\varepsilon'}}) \geq \frac{1}{1+C_{\delta}^2} (\Ll_{\H^2}(u|_{[a',r]})-1),
    \end{align}
    a contradiction to \cref{eq:lem-vert-1}! This shows the claim.
\end{proof}

\begin{proposition}\label{prop:vert}
    Under the assumptions of \cref{lem:vert}, one has
    \begin{equation}\label{eq:vert}
        \lim_{x\to b} \frac{\partial_xu^{(2)}(x)}{|\partial_xu(x)|} = \pm 1.
    \end{equation}
\end{proposition}
\begin{proof}
    Firstly, we show that
    \begin{equation}\label{eq:vert-1}
        \lim_{x\to b} \frac{|\partial_xu^{(2)}(x)|}{|\partial_xu(x)|} = 1.
    \end{equation}
    Suppose that \cref{eq:vert-1} fails. Using \cref{lem:vert}, it follows that there exist $\varepsilon>0$ and sequences $(x_n)_{n\in\N}$ and $(x_n')_{n\in\N}$ converging to $b$ with $x_n'<x_n<x_{n+1}'$ for all $n\in\N$ such that 
    \begin{equation}
        \frac{|\partial_xu^{(2)}(x_n)|}{|\partial_xu(x_n)|} \to 1\quad\text{and}\quad \frac{|\partial_xu^{(2)}(x_n')|}{|\partial_xu(x_n')|} \leq 1-\varepsilon\text{ for all $n\in\N$}.
    \end{equation}
    Fix now $n\in\N$. Using \cref{eq:app-br-gr}, if $\partial_xu^{(2)}(x_n)\geq 0$,  
    \begin{equation}
        \E(u|_{[x_n',x_n]}) \geq 4 \Bigl(\frac{\partial_xu^{(2)}(x_n)}{|\partial_xu(x_n)|} - \frac{\partial_xu^{(2)}(x_n')}{|\partial_xu(x_n')|}\Bigr) \geq 4 \Bigl( 1 + o(1) - \frac{|\partial_xu^{(2)}(x_n')|}{|\partial_xu(x_n')|} \Bigr) \geq 4\varepsilon + o(1).
    \end{equation}
    If however $\partial_xu^{(2)}(x_n)<0$, then
    \begin{equation}
        \E(u|_{[x_n,x_{n+1}']}) \geq 4 \Bigl(\frac{\partial_xu^{(2)}(x_{n+1}')}{|\partial_xu(x_{n+1}')|} - \frac{\partial_xu^{(2)}(x_n)}{|\partial_xu(x_n)|}\Bigr) \geq 4 \Bigl( 1 + o(1) - \frac{|\partial_xu^{(2)}(x_{n+1}')|}{|\partial_xu(x_{n+1}')|} \Bigr) \geq 4\varepsilon + o(1).
    \end{equation}
    Altogether, we have shown that
    \begin{equation}
        \E(u|_{[x_n',x_{n+1}']}) \geq 4\varepsilon + o(1)\quad\text{for $n\to\infty$}.
    \end{equation}
    Particularly, $\E(u)=\infty$, a contradiction! So \cref{eq:vert-1} is proved. The above argument now also yields \cref{eq:vert}. Indeed, if \cref{eq:vert} fails, then there are sequences $(x_n)_{n\in\N},(x_n')_{n\in\N}$ converging to $b$ with $x_n'<x_n<x_{n+1}'$ for all $n\in\N$ and
    \begin{equation}
        \frac{\partial_xu^{(2)}(x_n)}{|\partial_xu(x_n)|} \to 1\quad\text{and}\quad \frac{\partial_xu^{(2)}(x_n')}{|\partial_xu(x_n')|} \to -1.
    \end{equation}
    Using \cref{eq:app-br-gr}, one gets
    \begin{equation}
        \E(u|_{[x_n',x_n]}) \geq 4 \Bigl(\frac{\partial_xu^{(2)}(x_n)}{|\partial_xu(x_n)|} - \frac{\partial_xu^{(2)}(x_n')}{|\partial_xu(x_n')|}\Bigr) \to 8 \quad\text{for $n\to\infty$}
    \end{equation}
    which again yields $\E(u)=\infty$, a contradiction!
\end{proof}

Combining the verticality at ``hyperbolic infinity'' with \cref{eq:app-br-gr} yields lower bounds on the elastic energy which we state in the following Lemmata. 

\begin{lemma}\label{lem:conv-cl-2}
    Consider $u\colon [a,b)\to\H^2$ with $u\in W^{2,2}([a,r])$ an immersion for all $a<r<b$ where $b\in\R\cup\{\infty\}$ and suppose that $|u(x)|\to\infty$ for $x\to b$ while $\E(u)<\infty$. Then 
    \begin{equation}
       \E(u|_{[a,x]})\geq 4-4\frac{\partial_xu^{(2)}(a)}{|\partial_xu(a)|}+o(1)\quad\text{as $x\nearrow b$}.
    \end{equation} 
\end{lemma}
\begin{proof}
    Note that the condition $|u(x)|\to\infty$ for $x\to b$ ensures that $\Ll_{\H^2}(u)=\infty$.
    
    Next, we argue that $u^{(2)}$ is necessarily unbounded. Indeed, for contradiction, suppose that $u^{(2)}\leq R$. Fixing any sequence $x_n\nearrow b$ and denoting by $u_n$ a reparametrization of $\frac{1}{\Ll_{\R^2}(u|_{[a,x_n]})}u|_{[a,x_n]}$ on $[0,1]$, we have $\Ll_{\R^2}(u_n)=1$ for all $n\in\N$ and, since $\Ll_{\R^2}(u|_{[a,x_n]})\to\infty$, also $u_n^{(2)}\leq \frac{R}{\Ll_{\R^2}(u|_{[a,x_n]})}\to 0$ uniformly on $[0,1]$. Altogether, \cref{prop:elen-vs-euclen} yields that $\E(u_n)\to\infty$, i.e. $\E(u)=\infty$, a contradiction!

    Using \cref{prop:vert}, $\lim_{x\to b}\frac{\partial_xu^{(2)}(x)}{|\partial_xu(x)|} = \pm1$. If the limit equals $-1$, then there exists $a<r<b$ with $\partial_xu^{(2)}<0$ on $(r,b)$. Then one obtains that $u^{(2)}(x)\leq\max_{[a,r]}u^{(2)}<\infty$ for all $x\in[a,b)$, a contradiction to the above. Therefore,
    \begin{equation}
        \lim_{x\to b}\frac{\partial_xu^{(2)}(x)}{|\partial_xu(x)|} = 1.
    \end{equation}   
    With \cref{eq:app-br-gr},
    \begin{align}
        \E(u)\geq \E(u|_{[a,x]}) \geq 4\frac{\partial_xu^{(2)}}{|\partial_xu|}\Big|_a^{x} \to 4 - 4\frac{\partial_xu^{(2)}(a)}{|\partial_xu(a)|} \quad\text{for $x\to b$}.\qquad\qedhere
    \end{align}
\end{proof}

\begin{lemma}\label{lem:en-est-dd}
    Consider $u\colon [a,b)\to\H^2$ with $u\in W^{2,2}([a,r])$ an immersion for all $a<r<b$ where $b\in\R\cup\{\infty\}$ with $\E(u)<\infty$ and suppose that $u(x)\to (p^{(1)},0)^t$ for $x\to b$ for some $p^{(1)}\in\R$. Then 
    \begin{equation}
       \E(u|_{[a,x]})\geq -4-4\frac{\partial_xu^{(2)}(a)}{|\partial_xu(a)|}+o(1)\quad\text{as $x\nearrow b$}.
    \end{equation}
\end{lemma}
\begin{proof}
    Note that the condition $u(x)\to (p^{(1)},0)^t$ for $x\to b$ ensures that $\Ll_{\H^2}(u)=\infty$. By \cref{prop:vert}, with a similar argument as in \cref{lem:conv-cl-2}, using $u^{(2)}(x)\to 0$ for $x\nearrow b$, we have 
    \begin{equation}
        \lim_{x\to b} \frac{\partial_xu^{(2)}(x)}{|\partial_xu(x)|} = -1.
    \end{equation}
    This immediately yields the claim since, using \cref{eq:app-br-gr}, for any sequence $x_n\nearrow b$,
    \begin{align}
        \E(u)\geq \E(u|_{[a,x_n]}) \geq 4\frac{\partial_xu^{(2)}}{|\partial_xu|}\Big|_a^{x_n} = -4 - \frac{\partial_xu^{(2)}(a)}{|\partial_xu(a)|} + o(1).\qquad\qedhere
    \end{align}
\end{proof}

\subsection{Compactness}

The next lemma summarizes straight-forward consequences of $W^{2,2}$-bounds one obtains from bounds on the Willmore energy (or equivalently from bounds on the elastic energy by \cref{eq:br-gr}) as well as the quantization \cref{thm:ch-ve}.

\begin{lemma}\label{lem:compactness}
    Let $u_n\colon I\to\H^2$ be immersions on a compact interval $I\subseteq\R$. Suppose that $\E(u_n)\leq M$, $|u_n|\leq K$ and $|\partial_xu_n|\equiv L_n$ where $0<\ell\leq L_n\leq L$ for all $n\in\N$. Then there exists $u_{\infty}\in W^{1,\infty}(I,\R\times[0,\infty))$ such that, after passing to a subsequence, 
    \begin{equation}\label{eq:cpt-1}
        u_n\rightharpoonup^* u_{\infty}\text{ in }W^{1,\infty}\text{ and }u_n\to u_{\infty}\text{ uniformly.}
    \end{equation}
    Further, $Z=\{u_{\infty}^{(2)}=0\}$ is finite and, for any compact interval $J\subseteq I\setminus Z$, 
    \begin{equation}\label{eq:cpt-2}
        u_n|_J\rightharpoonup u_{\infty}|_J \quad\text{in $W^{2,2}(J,\R^2)$}.
    \end{equation}
\end{lemma}
\begin{proof}
    The assumptions immediately imply that $(u_n)_{n\in\N}\subseteq W^{1,\infty}(I,\R^2)$ is bounded such that, after passing to a subsequence, there exists $u_{\infty}\in W^{1,\infty}(I)$ with \cref{eq:cpt-1}. \cref{thm:ch-ve} then yields that $Z$ is finite with maximum number of elements only depending on $M$. Fix now any compact interval $J\subseteq I\setminus Z$.

    Since $u_{\infty}$ is continuous and $J$ compact, there exists $\varepsilon>0$ with $2\varepsilon < u_{\infty}^{(2)} < \frac{1}{2\varepsilon}$ on $J$. The convergence in \cref{eq:cpt-1} then yields that there exists $N_0\in\N$ such that, for $n\geq N_0$,
    \begin{equation}\label{eq:cpt-2.1}
        \varepsilon < u_n^{(2)} < \frac{1}{\varepsilon} \quad\text{on $J$}.
    \end{equation}
    Differentiating $|\partial_xu_n|\equiv L_n$, one obtains $\langle\partial_x^2u_n,\partial_xu_n\rangle = 0$ so that \cref{eq:ds2u} yields
    \begin{equation}
        \partial_s^2u_n = \Bigl(\frac{u_n^{(2)}}{L_n}\Bigr)^2 \partial_x^2u_n + u_n^{(2)}\frac{\partial_xu_n^{(2)}}{L_n} \frac{\partial_xu_n}{L_n}.
    \end{equation}
    Using $|v_1+v_2|^2\geq \frac12|v_1|^2-|v_2|^2$ and \cref{eq:cpt-2.1}, one thus estimates on $J$ 
    \begin{equation}\label{eq:cpt-3}
        |\partial_s^2u_n|_g^2 = \frac{1}{(u_n^{(2)})^2} |\partial_s^2u_n|^2 \geq \varepsilon^2 \left( \frac{\varepsilon^4}{2L^4} |\partial_x^2u_n|^2 -  \frac{1}{\varepsilon^2} \right) = \frac{\varepsilon^6}{2L^4} |\partial_x^2u_n|^2 - 1.
    \end{equation}
    Since
    \begin{equation}
        \left| \frac{1}{u_n^{(2)}} \begin{pmatrix}
            -2 \partial_su_n^{(1)}\partial_su_n^{(2)}\\
            (\partial_su_n^{(1)})^2-(\partial_su_n^{(2)})^2
        \end{pmatrix} \right|_g^2 = |\partial_su_n|_g^4=1,
    \end{equation}
    \cref{eq:curv,eq:cpt-3,eq:cpt-2.1} yield
    \begin{align}
        M&\geq \E(u_n)\geq \int_J |\curv_n|_g^2\dd s \geq \int_J \bigl(\frac{\varepsilon^6}{4L^4} |\partial_x^2u_n|^2 - \frac32\bigr) \frac{|\partial_xu_n|}{u_n^{(2)}}\dd x \\
        &\geq  \frac{\varepsilon^7\ell}{4L^4} \|\partial_x^2u_n\|_{L^2(J)}^2 - \frac32 \frac{L|J|}{\varepsilon},
    \end{align}
    using again $|v_1+v_2|^2\geq \frac12|v_1|^2-|v_2|^2$. So $(u_n|_J)_{n\in\N}\subseteq W^{2,2}(J)$ is bounded which implies \cref{eq:cpt-2}, using \cref{eq:cpt-1} and a subsequence argument. 
\end{proof}


\section{Proof of the main results}\label{sec:pf-main}

A major ingredient in our proof is the following statement on boundedness. Particularly, any family of cylindrical surfaces of revolution satisfying fixed Dirichlet boundary data whose Willmore energies remain uniformly below $8\pi$ has uniformly bounded diameter. 

As it turns out, a similar observation is crucial in proving the existence of minimizers of the Willmore energy with prescribed Dirichlet boundary data in \cite[Theorem 4.1]{novagapozzetta2020}. While Novaga-Pozzetta have the analog of \cref{prop:bd-euc-len} not only for cylindrical surfaces of revolution but also for general integral varifolds satisfying Dirichlet boundary data, their result requires that the Willmore energies remain uniformly below $4\pi$.

\begin{proposition}\label{prop:bd-euc-len}
    Let $u_n\colon[0,1]\to\H^2$ be immersions satisfying
    \begin{equation}\label{eq:prop-en}
        \limsup_{n\to\infty}\W(f_{u_n})< 8\pi.
    \end{equation}
    Further, assume the boundary conditions
    \begin{equation}\label{eq:prop-dn}
        u_n(y)=p_y\quad\text{and}\quad\frac{\partial_xu_n(y)}{|\partial_xu_n(y)|}=\tau_y\quad\text{for all $y\in\{0,1\}$}
    \end{equation}
    where $p_0,p_1\in\H^2$ and $\tau_y\in\S^1$, $y=0,1$, are given. Then $\Ll_{\R^2}(u_n)$ and thus also $\mathrm{diam}(f_{u_n})$ are uniformly bounded in $n\in\N$.
\end{proposition}
\begin{proof}
    Suppose that the statement fails. That is, using also \cref{eq:prop-dn}, after passing to a subsequence without relabeling, we have w.l.o.g. $L_n\vcentcolon=\Ll_{\R^2}(u_n)\to\infty$ for $n\to\infty$ and, using \cref{eq:prop-en}, $\W(f_{u_n})\leq M<8\pi$. 

    Write $\varphi_n^{[0]}\colon  [0,L_n]\to [0,1]$ for the inverse of the map given by $y\mapsto \int_0^y|\partial_xu_n|\dd x$. Further, $\varphi_n^{[1]}\colon[0,L_n]\to[0,1]$ is given by the inverse of $y\mapsto \int_y^1 |\partial_xu_n|\dd x$. Then define $u_n^{[i]} = u_n\circ\varphi_n^{[i]}$ for $i=0,1$ and $n\in\N$.

    Using the boundary conditions in \cref{eq:prop-dn}, we have 
    \begin{equation}
        u_n^{[0]}(0)=u_n(0)=p_0,\ u_n^{[1]}(0)=u_n(1)=p_1\quad\text{for all $n\in\N$}.
    \end{equation}
    For $R\in\N$, one has $[0,R]\subseteq[0,L_n]$ for $n\geq N(R)$. By the above, for each $R\in\N$, $(u_n^{[i]}|_{[0,R]})_{n\geq N(R)}$ is uniformly bounded in $W^{1,\infty}([0,R])$. 

    By choosing a suitable diagonal sequence, one obtains that, after passing to subsequence without relabeling, there exists $u_{\infty}^{[i]}\colon[0,\infty)\to \R\times[0,\infty)$ such that, for all $R\in\N$,
    \begin{equation}
        u_n^{[i]}|_{[0,R]} \rightharpoonup^* u_{\infty}^{[i]}|_{[0,R]}\text{ in $W^{1,\infty}([0,R])$ and }u_n^{[i]}|_{[0,R]} \to u_{\infty}^{[i]}|_{[0,R]}\text{ uniformly}.
    \end{equation}
    Further, by \cref{eq:br-gr}, \cref{eq:prop-en,eq:prop-dn}, 
    \begin{equation}\label{eq:diam-0}
        \E(u_n) = \frac{2}{\pi}\W(f_{u_n}) + 4{\tau_y^{(2)}}\Big|_0^1 \leq \frac{2M}{\pi} + 8 < 24.
    \end{equation}
    Using \cref{lem:compactness} and \cref{thm:ch-ve}, writing $Z_i=\{(u_{\infty}^{[i]})^{(2)}=0\}$, it follows that $Z_i$ consists of at most two points and, for any compact interval $K\subseteq [0,\infty)\setminus Z_i$, one has 
    \begin{equation}
        u_n^{[i]}|_K \rightharpoonup u_{\infty}^{[i]}|_K\quad\text{in $W^{2,2}(K)$}.
    \end{equation}
    Consider only $R>0$ with $(Z_1\cup Z_2)\cap (R,\infty)=\emptyset$. If $n\in\N$ is sufficiently large such that $L_n>2R$, then $\varphi_n^{[0]}(R)<\varphi_n^{[1]}(R)$. Indeed, if $\varphi_n^{[0]}(R)\geq \varphi_n^{[1]}(R)$, then
    \begin{equation}
        L_n=\Ll_{\R^2}(u_n) \leq \Ll_{\R^2}(u_n|_{[0,\varphi_n^{[0]}(R)]}) + \Ll_{\R^2}(u_n|_{[\varphi_n^{[1]}(R),1]}) = 2R,
    \end{equation}
    a contradiction. Writing $\alpha_n(R)=\varphi_n^{[0]}(R)$ and $\beta_n(R)=\varphi_n^{[1]}(R)$, we further obtain that $\lim_{n\to\infty}\Ll_{\H^2}(u_n|_{[\alpha_n(R),\beta_n(R)]})=\infty$. Indeed, since $(Z_1\cup Z_2)\cap (R,\infty)=\emptyset$, one has $u_n(\alpha_n(R))\to u_{\infty}^{[0]}(R)\in\H^2$, $u_n(\beta_n(R))\to u_{\infty}^{[1]}(R)\in\H^2$. \\Moreover, $\Ll_{\R^2}(u_n|_{[\alpha_n(R),\beta_n(R)]})=L_n-2R\to\infty$ so that, using \cref{prop:elen-vs-euclen} and $\E(u_n)<24$ by \cref{eq:diam-0}, $\lim_{n\to\infty}\max_{[\alpha_n(R),\beta_n(R)]}u_n^{(2)}=\infty$. \\Combined, this gives $\lim_{n\to\infty}\Ll_{\H^2}(u_n|_{[\alpha_n(R),\beta_n(R)]})=\infty$. Therefore, \cref{rem:eich-grun} yields 
    \begin{equation}\label{eq:diam-1}
        \liminf_{n\to\infty}\E(u_n|_{[\alpha_n(R),\beta_n(R)]}) \geq 8.
    \end{equation}
    \textbf{Case 1:} $Z_i\neq\emptyset$. Denote by $x^*$ the smallest element in $Z_i$. Let $\eta>0$. Since $u_n^{[i]}(x^*\pm\eta) \to u_{\infty}^{[i]}(x^*\pm\eta)\in\H^2$ and as $(u_n^{[i]})^{(2)}(x^*)\to 0$, we have $\lim_{n\to\infty}\Ll_{\H^2}(u_n^{[i]}|_{[x^*-\eta,x^*+\eta]})=\infty$, so that, by \cref{rem:eich-grun}, 
    \begin{equation}\label{eq:diam-2}
        \liminf_{n\to\infty}\E(u_n^{[i]}|_{[x^*-\eta,x^*+\eta]}) \geq 8.
    \end{equation}
    Moreover, \cref{lem:en-est-dd} implies
    \begin{equation}
        \E(u_{\infty}^{[i]}|_{[0,x^*-\eta]}) \geq -4 + 4(-1)^{i+1} {\tau_i^{(2)}} + o(1)\quad\text{as $\eta\searrow 0$}.
    \end{equation}  
    Combined with \cref{eq:diam-2}, since $\liminf_{n\to\infty}\E(u_n^{[i]}|_{[0,x^*-\eta]})\geq \E(u_{\infty}^{[i]}|_{[0,x^*-\eta]})$ due to weak convergence in $W^{2,2}$, c.f. \cref{lem:elen-semi-cont}, one obtains
    \begin{equation}\label{eq:sec-ingr-0}
        \liminf_{n\to\infty} \E(u_{n}^{[i]}|_{[0,R]}) \geq 4+4(-1)^{i+1}{\tau_i^{(2)}} .
    \end{equation}
    \textbf{Case 2:} $Z_i=\emptyset$. Then 
    \begin{equation}\label{eq:sec-ingr-1}
        \liminf_{n\to\infty} \E(u_{n}^{[i]}|_{[0,R]}) \geq 4+4(-1)^{i+1}{\tau_i^{(2)}} + o(1)\quad\text{for $R\to\infty$}
    \end{equation}
    simply follows by \cref{lem:conv-cl-2} applied to $u_{\infty}^{[i]}$ and $\liminf_{n\to\infty}\E(u_n^{[i]}|_{[0,R]})\geq \E(u_{\infty}^{[i]}|_{[0,R]})$.

    Altogether, combining \cref{eq:diam-1,eq:sec-ingr-0,eq:sec-ingr-1}, one obtains
    \begin{align}
        \liminf_{n\to\infty}\E(u_n) &\geq \liminf_{n\to\infty}\E(u_n|_{[0,\alpha_n(R)]}) + \liminf_{n\to\infty}\E(u_n|_{[\alpha_n(R),\beta_n(R)]}) + \liminf_{n\to\infty}\E(u_n|_{[\beta_n(R),1]}) \\
        &\geq \liminf_{n\to\infty}\E(u_{n}^{[0]}|_{[0,R]}) + 8 + \liminf_{n\to\infty}\E(u_{n}^{[1]}|_{[0,R]})\\
        &\geq 16 + 4{\tau_y^{(2)}}\Big|_{y=0}^1.
    \end{align}
    That is, using \cref{eq:br-gr}, $8\pi>\liminf_{n\to\infty}\W(f_{u_n})\geq 8\pi$, a contradiction!
\end{proof}

Using \cref{prop:euc-len-diam,lem:compactness}, considering for instance a minimizing sequence in the context of \cref{thm:main-calvar}, after choosing a suitable subsequence, one can pass to a limit (in a suitable sense). As it turns out, this limit does not intersect the rotation axis if the Willmore energies along the sequence are bounded away from the Li-Yau constant $C_{\mathrm{LY}}^{\mathrm{rot}}(p,\tau)$ induced by the Dirichlet boundary conditions. Roughly speaking, this is due to fact that one can always pass to a weak varifold limit which, if it intersects the rotation axis, always does so in a point with density at least two. This of course violates the Li-Yau inequality \cref{prop:Li-Yau}. The rigorous arguments are given in 

\begin{proposition}\label{prop:prop2-bd-hyp-len}
    Consider a family of immersions $u_n\colon[0,1]\to\H^2$ satisfying the boundary conditions in \cref{eq:prop-dn} with
    \begin{equation}\label{eq:prop2-thresh}
        \limsup_{n\to\infty} \W(f_{u_n}) < C_{\mathrm{LY}}^{\mathrm{rot}}(p,\tau).
    \end{equation}
    Then $\sup_{n\in\N}\Ll_{\H^2}(u_n)<\infty$.
\end{proposition}
\begin{proof}
    Suppose that the claim fails. W.l.o.g., after passing to a subsequence without relabeling, we then have 
    \begin{equation}\label{eq:prop2-2}
        \lim_{n\to\infty}\Ll_{\H^2}(u_n) = \infty\quad\text{and}\quad \sup_{n\in\N}\W(f_{u_n}) < C_{\mathrm{LY}}^{\mathrm{rot}}(p,\tau).
    \end{equation}    
    Using \cref{prop:bd-euc-len,rem:cly-vs-8pi}, one obtains $\Ll_{\R^2}(u_n)\leq L$ for all $n\in\N$. Thus, consider the reparametrizations $\widetilde{u_n}\colon[0,1]\to\H^2$ of $u_n$ by constant Euclidean speed. As in \cite[p.9]{schlierf2024}, one argues that $\inf_{n\in\N}\Ll_{\H^2}(u_n)>0$ and then, using the boundary conditions, one obtains $\Ll_{\R^2}(u_n)\geq\ell>0$ similarly as in \cite[Remark 3.2]{schlierf2023}.

    By \cref{lem:compactness}, after passing to a subsequence, we have that $\widetilde{u_n}\to \widetilde{u_{\infty}}$ uniformly and weakly* in $W^{1,\infty}([0,1])$. Moreover, using \cref{thm:ch-ve} and $\limsup_{n\to\infty}\E(\widetilde{u_n})<24$, cf. \cref{eq:diam-0}, $Z=\{\widetilde{u_{\infty}}^{(2)}=0\}$ consists of at most two elements. Again by \cref{lem:compactness}, $\widetilde{u_n}\rightharpoonup\widetilde{u_{\infty}}$ in $W^{2,2}$ (and thus strong convergence in $C^1$) on compact intervals $J$ in $[0,1]\setminus Z$. Particularly, writing $|\partial_x\widetilde{u_n}|\equiv L_n$, we obtain $|\partial_x\widetilde{u_{\infty}}|\equiv L$ on $[0,1]\setminus Z$ where $L=\lim_{n\to\infty}L_n$ due to the $C^1$ convergence locally in $[0,1]\setminus Z$.

    Now considering the varifolds $V_n=V_{f_{\widetilde{u_n}}}$ induced by the immersions $f_{\widetilde{u_n}}$, we have after passing to a further subsequence without relabeling $V_n\rightharpoonup V_{\infty}$ for $n\to\infty$ weakly as varifolds, using \cref{thm:vari-cpt}. Due to lower semi-continuity of $\W$ with respect to weak varifold convergence and \cref{eq:prop2-thresh},
    \begin{equation}
        \W(V_{\infty}) \leq \liminf_{n\to\infty}\W(V_n) =  \liminf_{n\to\infty}\W(f_{\widetilde{u_n}}) < C_{\mathrm{LY}}^{\mathrm{rot}}(p,\tau).
    \end{equation}
    As in \cite[pp. 539 -- 540]{novagapozzetta2020}, one argues that the constant boundary data $\sigma_{V_n}=\sigma_{V_{1}}$ and $\eta_{V_n}=\eta_{V_1}$ for all $n\in\N$ is stable with respect to weak varifold convergence. Therefore, using \cref{rem:p-tau-for-vari}, \cref{prop:Li-Yau} applied to $V_{\infty}$ yields in particular that
    \begin{equation}\label{eq:pf-main-2}
        \theta^2(\mu_{V_{\infty}},z)<2\quad\text{for all $z\in\R\times\{0\}\times\{0\}$}.
    \end{equation}    
    Suppose that $Z\neq\emptyset$, that is, there exists $x^*\in Z$ with $\widetilde{u_{\infty}}^{(2)}(x^*)=0$. Writing $z^*=(\widetilde{u_{\infty}}(x^*),0)\in\R\times\{0\}\times\{0\}$, we claim that the density of $\mu_{V_{\infty}}$ at $z^*$ is at least $2$.

    To show this, we proceed as follows. For $\rho>0$, fix $a_{\rho}<x^*<b_{\rho}$ such that 
    \begin{equation}\label{eq:arho-brho-xstar}
        |\widetilde{u_{\infty}}(a_{\rho})-\widetilde{u_{\infty}}(x^*)|=|\widetilde{u_{\infty}}(b_{\rho})-\widetilde{u_{\infty}}(x^*)|=\rho\quad\text{and}\quad |\widetilde{u_{\infty}}-\widetilde{u_{\infty}}(x^*)| < \rho\text{ on $(a_{\rho},b_{\rho})$}.
    \end{equation}
    Using \cref{prop:vert} and $|\partial_x\widetilde{u_{\infty}}|\equiv L$ on $[0,1]\setminus Z$ for some $L>0$, it follows that $\lim_{x\to x^*}\partial_x\widetilde{u_{\infty}}^{(1)}(x)=0$.
    
    \textbf{Claim:} One has $\widetilde{u_{\infty}}^{(1)}(a_{\rho})-\widetilde{u_{\infty}}^{(1)}(x^*)=o(\rho)$ and $\widetilde{u_{\infty}}^{(1)}(b_{\rho})-\widetilde{u_{\infty}}^{(1)}(x^*)=o(\rho)$.

    Indeed, fix any $\varepsilon>0$. Then there exists $\eta>0$ such that, for all $x\in (x^*-\eta,x^*+\eta)$ with $x\neq x^*$, $|\partial_x\widetilde{u_{\infty}}^{(1)}(x)|<\varepsilon$. 

    Observe that, using \cref{prop:euc-len-diam,eq:diam-0}
    \begin{equation}\label{eq:pf-mr-1}
        |a_{\rho}-x^*| = \frac{1}{L}\int_{a_{\rho}}^{x^*}|\partial_x\widetilde{u_{\infty}}|\dd x \leq \frac{1}{L} \Ll_{\R^2}(\widetilde{u_{\infty}}|_{(a_{\rho},x^*)}) \leq \frac{C}{L} \rho.
    \end{equation}
    Similarly, also $|b_{\rho}-x^*|\leq \frac{C}{L} \rho$. Thus, $x^*-\eta<a_{\rho}<x^*<b_{\rho}<x^*+\eta$ for $\rho>0$ sufficiently small. For such $\rho$,  
    \begin{align}
        \frac{1}{\rho} |\widetilde{u_{\infty}}^{(1)}(a_{\rho})-\widetilde{u_{\infty}}^{(1)}(x^*)|\leq \frac{1}{\rho} \int_{a_{\rho}}^{x^*}|\partial_x\widetilde{u_{\infty}}^{(1)}|\dd x\leq \frac{x^*-a_{\rho}}{\rho} \varepsilon\leq \frac{C}{L} \varepsilon,
    \end{align} 
    using \cref{eq:pf-mr-1}. One proceeds similarly for $b_{\rho}$. This shows the above claim.

    Let now $\varepsilon>0$. By the claim, there exists $\rho_0(\varepsilon)>0$ sufficiently small such that, for all $0<\rho<\rho_0(\varepsilon)$, 
    \begin{equation}
        \frac{1}{\rho} |\widetilde{u_{\infty}}^{(1)}(a_{\rho})-\widetilde{u_{\infty}}^{(1)}(x^*)|, \frac{1}{\rho} |\widetilde{u_{\infty}}^{(1)}(b_{\rho})-\widetilde{u_{\infty}}^{(1)}(x^*)| < \varepsilon.
    \end{equation}
    Since $\rho^2=(\widetilde{u_{\infty}}^{(1)}(y)-\widetilde{u_{\infty}}^{(1)}(x^*))^2+\widetilde{u_{\infty}}^{(2)}(y)^2$ for $y\in\{a_{\rho},b_{\rho}\}$, we obtain
    \begin{equation}
        \widetilde{u_{\infty}}^{(2)}(y) \geq \sqrt{1-\varepsilon^2} \rho\quad \text{for $y\in\{a_{\rho},b_{\rho}\}$}.
    \end{equation}
    This yields 
    \begin{equation}
        \int_{a_{\rho}}^{x^*}\widetilde{u_{\infty}}^{(2)} L \dd x\geq -\int_{a^{\rho}}^{x^*}\widetilde{u_{\infty}}^{(2)} \partial_x\widetilde{u_{\infty}}^{(2)}\dd x \geq \int_0^{\sqrt{1-\varepsilon^2} \rho} h\dd h = \frac12 (1-\varepsilon^2)\rho^2.
    \end{equation}
    One proceeds similarly on $[x^*,b_{\rho}]$ so that altogether, using the uniform convergence $\widetilde{u_n}\to\widetilde{u_{\infty}}$ and $|\partial_x\widetilde{u_n}|\equiv L_n\to L$,
    \begin{equation}
        \lim_{n\to\infty} \int_{a_{\rho}}^{b_{\rho}} \widetilde{u_n}^{(2)}|\partial_x\widetilde{u_n}|\dd x = \int_{a_{\rho}}^{b_{\rho}}\widetilde{u_{\infty}}^{(2)} L \dd x \geq (1-\varepsilon^2)\rho^2.
    \end{equation}
    By \cref{eq:arho-brho-xstar}, $f_{\widetilde{u_n}}((a_{\rho},b_{\rho})\times \S^1)\subseteq B_{\rho}(z^*)$, so that the above results in
    \begin{equation}\label{eq:pf-main-1}
        \limsup_{n\to\infty}\frac{\mu_{V_n}(B_{\rho}(z^*))}{\pi\rho^2} \geq  \limsup_{n\to\infty} \frac{1}{\pi\rho^2} 2\pi\int_{a_{\rho}}^{b_{\rho}} \widetilde{u_n}^{(2)}|\partial_x\widetilde{u_n}|\dd x \geq 2(1-\varepsilon^2),
    \end{equation}
    for $0<\rho<\rho_0(\varepsilon)$. By varifold convergence $V_n\rightharpoonup V_{\infty}$ one especially obtains that $\mu_{V_n}\rightharpoonup^* \mu_{V_{\infty}}$ as Radon measures so that, using \cref{eq:pf-main-1},
    \begin{equation}
        (1+\varepsilon)^2\theta^2(\mu_{V_{\infty}},z^*)=\lim_{\rho\searrow 0}\frac{\mu_{V_{\infty}}(B_{(1+\varepsilon)\rho}(z^*))}{\pi \rho^2} \geq \limsup_{\rho\searrow 0}\limsup_{n\to\infty} \frac{\mu_{V_n}(B_{\rho}(z^*))}{\pi\rho^2} \geq 2(1-\varepsilon^2),
    \end{equation}
   contradicting \cref{eq:pf-main-2} for sufficiently small $\varepsilon>0$.

   We obtain $Z=\emptyset$. Therefore, as $\widetilde{u_{\infty}}$ is continuous, there is $\varepsilon>0$ with $2\varepsilon<\widetilde{u_{\infty}}^{(2)} <\frac{1}{2\varepsilon}$. The uniform convergence yields $\varepsilon<\widetilde{u_n}<\frac{1}{\varepsilon}$ for $n$ sufficiently large such that
   \begin{align}
    \Ll_{\H^2}(u_n) &= \Ll_{\H^2}(\widetilde{u_n}) = \int_I \frac{|\partial_x\widetilde{u_n}|}{\widetilde{u_n}^{(2)}} \dd x\leq \frac{L}{\varepsilon}.
   \end{align}
   This contradicts \cref{eq:prop2-2}!
\end{proof}
\begin{proof}[Proof of \cref{thm:main-calvar}.]
    Consider a minimizing sequence $(u_n)_{n\in\N}$, i.e. each $u_n$ satisfies \cref{eq:bdry-cdts} and $\W(f_{u_n}) \to M_{p,\tau}$. Since $M_{p,\tau}<C_{\mathrm{LY}}^{\mathrm{rot}}(p,\tau)$, we may w.l.o.g. suppose that
    \begin{equation}
        \sup_{n\in\N} \W(f_{u_n}) < C_{\mathrm{LY}}^{\mathrm{rot}}(p,\tau).
    \end{equation}
    Using \cref{prop:prop2-bd-hyp-len} and \cref{eq:bdry-cdts}, one finds that $\Ll_{\H^2}(u_n)\leq C$. As in \cite[Theorem 3.1]{eichmanngrunau2019}, one concludes that there exists a minimizing immersion $u\in W^{2,2}([0,1],\H^2)$ satisfying \cref{eq:bdry-cdts} which is parametrized by constant speed in $\H^2$.

    Therefore, for any $\varphi\in C_c^{\infty}((0,1),\R^2)$,
    \begin{equation}
        \frac{\dd}{\dd t}\E({u} + t\varphi) \Big|_{t=0} = 0.
    \end{equation}
    As in \cite[Section 5]{eichmanngrunau2019}, one concludes ${u}\in C^{\infty}([0,1])$ and ${u}$ satisfies the point-wise Euler-Lagrange equation
    \begin{equation}\label{eq:main-calvar-pf}
        \nabla\E({u}) = 0 \quad\text{in $[0,1]$}.
    \end{equation}
    Using \cite[Theorem 4.1]{dallacquaspener2018} and \cref{eq:main-calvar-pf}, $f_u$ is a Willmore surface of revolution.
\end{proof}

\begin{proof}[Proof of \cref{thm:main}.]
   Let $u\colon[0,T)\times[0,1]\to\H^2$ be a maximal solution to \cref{eq:wf-eq}. Using \cref{eq:l2-grad-will,eq:dir-will-flow}, one finds
   \begin{equation}
    \frac{\dd}{\dd t}\W(f_{u(t)}) = -\int_C |\partial_tf_u|^2\dd\mu_{f_u} = -\int_C |\nabla\W(f_u)|^2\dd\mu_{f_u}  \leq 0.
   \end{equation}
   Therefore, either the flow is stationary, i.e. $f_{u_0}$ is a Willmore surface and the flow satisfies $T=\infty$ and $f_{u(t)}=f_{u_0}$ for all $t\in[0,\infty)$, or $\frac{\dd}{\dd t}\W(f_{u(t)})\big|_{t=0}<0$. 
   
   For this reason, we may w.l.o.g. suppose $\sup_{t\in[\delta,T)}\W(f_{u(t)})<\W(f_{u_0})\leq C_{\mathrm{LY}}^{\mathrm{rot}}(p,\tau)$ for some $\delta\in (0,T)$. Especially, \cref{prop:prop2-bd-hyp-len} applies and yields that $\Ll_{\H^2}(u(t))$ is uniformly bounded in $t\in[\delta,T)$. As $u\in C^1([0,\delta]\times[0,1])$, we even get $\sup_{t\in[0,T)}\Ll_{\H^2}(u(t))<\infty$. Thus, the claim is proved, using \cite[Theorems 3.14 and 4.5]{schlierf2024}. 
\end{proof}


\appendix

\section{Immersions into $\H^2$ with Sobolev regularity}

The above definitions of the curvature, tangent vector and elastic energy also make sense for immersions $u\in W^{2,2}(I,\H^2)$ where $I\subseteq\R$ is a compact interval. In this article, $W^{2,2}(I,\H^2)\vcentcolon= \{v\in W^{2,2}(I,\R^2):v(I)\subseteq\H^2\}$ is to be understood as an open subset of the Sobolev space $W^{2,2}(I,\R^2)\hookrightarrow C^1(I,\R^2)$. 

A curve $u\in W^{2,2}(I,\H^2)$ is called \emph{immersion} if $|\partial_xu|>0$. For such $u$, $\partial_s=1/|\partial_xu|\partial_x$ denotes the (weak) derivative with respect to the arc-length parameter $s$. By the chain-rule for Sobolev functions, 
\begin{equation}\label{eq:ds2u}
    \partial_s^2u = \Bigl(\frac{u^{(2)}}{|\partial_xu|}\Bigr)^2\partial_x^2u - \langle \Bigl(\frac{u^{(2)}}{|\partial_xu|}\Bigr)^2\partial_x^2u,\frac{\partial_xu}{|\partial_xu|}\rangle\frac{\partial_xu}{|\partial_xu|} + u^{(2)}\frac{\partial_xu^{(2)}}{|\partial_xu|}\frac{\partial_xu}{|\partial_xu|} \in L^2(I,\R^2)
\end{equation}
exists and as in \cite[(12)]{dallacquaspener2017}, one computes that
\begin{equation}\label{eq:curv}
    \curv = \nabla_{\partial_su}\partial_su = \partial_s^2u + \frac{1}{u^{(2)}} \begin{pmatrix}
        -2\partial_su^{(1)}\partial_su^{(2)}\\(\partial_su^{(1)})^2-(\partial_su^{(2)})^2
    \end{pmatrix}.
\end{equation}
Moreover, as already computed in \cite[(2.9)]{eichmanngrunau2019},
\begin{equation}\label{eq:elen-w22}
    \E(u) =\int_0^1 \bigl( \partial_x^2u^{(2)}\partial_xu^{(1)}u^{(2)}-\partial_x^2u^{(1)}\partial_xu^{(2)}u^{(2)}+\partial_xu^{(1)}|\partial_xu|^2 \bigr)^2\frac{1}{u^{(2)}|\partial_xu|^5}\dd x.
\end{equation}

\begin{lemma}\label{lem:elen-cont}
    Consider immersions $u_n,u\in W^{2,2}(I,\H^2)$ for $n\in\N$ with $u_n\to u$ in $W^{2,2}(I)$. Then $\E(u_n)\to\E(u)$.
\end{lemma}
\begin{proof}
    The claim is due to \cref{eq:elen-w22} and the compact embedding $W^{2,2}(I)\hookrightarrow C^1(I,\R^2)$.
\end{proof}

\begin{lemma}\label{lem:elen-semi-cont}
    For immersions $u_n,u\in W^{2,2}(I,\H^2)$ for $n\in\N$ with $u_n\rightharpoonup u$ in $W^{2,2}(I)$, one has
    \begin{equation}
        \E(u)\leq\liminf_{n\to\infty}\E(u_n).
    \end{equation}
\end{lemma}
\begin{proof}
    After passing to a subsequence, w.l.o.g. $u_n\to u$ in $C^1(I)$ and $\lim_{n\to\infty}\E(u_n)=\liminf_{n\to\infty}\E(u_n)$. It is a well-known fact that the product of a uniformly convergent sequence and a sequence that weakly converges in $L^2$ is again weakly convergent in $L^2$ to the product of the respective limits. Thus, the claim is a consequence of \cref{eq:elen-w22} and the weak lower semi-continuity of $\|\cdot\|_{L^2(I)}$.
\end{proof}


\section{Energy estimates}

We use the following two technical results from \cite{schlierf2023}.

\begin{proposition}\label{prop:elen-vs-euclen}
    Consider immersions $u_n\colon I\to\H^2$ with $u_n\in W^{2,2}(I)$ defined on a compact interval $I\subseteq\R$ such that $\|u_n^{(2)}\|_{L^{\infty}(I)}\to 0$ and $\Ll_{\R^2}(u_n)\geq \ell>0$. Then $\limsup_{n\to\infty}\E(u_n)=\infty$.
\end{proposition}
\begin{proof}
    \cite[Proposition 2.7]{schlierf2023} yields the claim for smooth immersions $u_n$. Using \cite{meyersserrin1964} to approximate each $u_n\in W^{2,2}(I)$ suitably and using \cref{lem:elen-cont}, one obtains the claim.
\end{proof}

\begin{theorem}[{\cite[Theorem 2.5]{schlierf2023}}]\label{thm:ch-ve}
    Consider smooth immersions $u_n\colon I\to\H^2$ on a compact interval $I\subseteq\R$ such that $0<\ell\leq |\partial_x u_n|\leq L<\infty$. If $(|\partial_xu_n|)_{n\in\N}$ converges uniformly on $I$,
    \begin{equation}
        u_n\rightharpoonup^*u_{\infty}\text{ in $W^{1,\infty}(I)$}\quad\text{and}\quad u_n\to u_{\infty}\text{ uniformly on $I$}
    \end{equation}
    and
    \begin{equation}\label{eq:ch-ve-1}
        \sup_{n\in\N}\E(u_n) \leq E\cdot 8 + \eta\quad\text{for some $E\in\N_0$ and $\eta<8$},
    \end{equation}
    then $Z=\{u_{\infty}^{(2)}=0\}\cap I^{\circ}$ consists of at most $E$ points.
\end{theorem}


\section{Monotonicity formula and estimates for $C_{\mathrm{LY}}(p,\tau)$}
\label{app:est-cly}

As it turns out, the underlying monotonicity formula in \cite[Appendix B]{novagapozzetta2020} from which \cref{prop:Li-Yau} is concluded gives very useful exact identities for surfaces contained in spheres. Since the integral in \cref{eq:Li-Yau} only depends on the Dirichlet boundary conditions and as the Dirichlet boundary conditions of cylindrical surfaces of revolution are parallel circles with prescribed co-normal field, $C_{\mathrm{LY}}(p,\tau)$ can be estimated using suitable cut-outs of spheres.

First, recall Simon's monotonicity formula in \cite[Appendix B]{novagapozzetta2020} for varifolds with boundary, cf. \cite[(1.2)]{simon1993} and \cite[Appendix A]{kuwertschaetzle2004} for the versions without boundary. Consider a $2$-dimensional integral varifold $V$ in $\R^3$ with generalized mean curvature $\vec{H}$ and $\W(V),\mu_V(\R^3)<\infty$. For any $z\in\R^3$ such that
\begin{equation}
    \int_{B_1(z)} \frac{|\langle x-z,\eta_V(x)\rangle|}{|x-z|^2}\dd\sigma_V(x)<\infty,
\end{equation}
writing $\bot$ for the orthogonal projection onto the orthogonal complement of the approximate tangent space, one has
\begin{equation}\label{eq:mon-form}
    A_z(\sigma) + \int_{B_{\rho}(z)\setminus B_{\sigma}(z)} \Bigl|\frac{\vec{H}(x)}{2} + \frac{(x-z)^{\bot}}{|x-z|^2}\Bigr|^2\dd\mu_V(x) = A_z(\rho)  
\end{equation}
for all $0<\sigma<\rho$ where, for $t>0$,
\begin{align}
    A_z(t) &= \frac{\mu_V(B_t(z))}{t^2} + \frac14\int_{B_t(z)} |\vec{H}|^2\dd\mu_V + \int_{B_t(z)} \frac{\langle\vec{H}(x),x-z\rangle}{t^2}\dd\mu_V(x) \\
    &\quad+ \frac12 \int_{B_{t}(z)} \Bigl(\frac{1}{|x-z|^2}-\frac{1}{t^2}\Bigr)\langle x-z,\eta_V(x)\rangle\dd\sigma_V(x).
\end{align}
For $y=0,1$, consider points $p_y\in\H^2$ and unit tangent vectors $\tau_y\in\S^1$ which aren't vertical, i.e. $\tau_y^{(1)}\neq 0$. The result below remains true in the cases where $\tau_y=\pm(0,1)^t$ is allowed and can be proved similarly by using cut-outs of planes instead of spheres. Recall the energy threshold $T(\tau)=4\pi-2\pi \tau_y^{(2)}\big|_{y=0}^1$ in \cite{schlierf2024}. 
\begin{proposition}\label{prop:battle-of-thresholds}
    One has that
    \begin{equation}
        C_{\mathrm{LY}}(p,\tau) \geq T(\tau).
    \end{equation}
\end{proposition}
To this end, we make the following construction. First, consider the parallel circles
\begin{equation}
    C_y = \{(p_y^{(1)},p_y^{(2)}\cos(\theta),p_y^{(2)}\sin(\theta))^t:\theta\in\R\}\quad\text{for $y=0,1$}.
\end{equation}
Consider now an integer rectifiable $2$-varifold $V_y$ with generalized mean curvature $\vec{H}_y$, finite mass and Willmore energy such that, for
\begin{equation}
    \eta_y(p_y^{(1)},p_y^{(2)}\cos\theta,p_y^{(2)}\sin\theta) = (-1)^{1+y} \bigl(\tau_y^{(1)},\tau_y^{(2)}\cos\theta,\tau_y^{(2)}\sin\theta\bigr)^t\quad\text{for $\theta\in\R$},
\end{equation}
one has $\sigma_{V_y}=\Haus^1\measurerestr C_y$ and $\eta_{V_y}=\eta_y$. Using \cite[(B.8)]{novagapozzetta2020} for the case $z\notin C_y$ and \cite[(4.2)]{novagapozzetta2020} for $z\in C_y$, one finds
\begin{equation}
    \lim_{t\searrow 0} A_z(t) = \theta^2(\mu_{V_y},z)\pi\quad\text{for all $z\in\R^3$}.
\end{equation}
Moreover, estimating $\int_{B_t(z)} \frac{\langle\vec{H}_y,x-z\rangle}{t^2}\dd\mu_{V_y}$ as in \cite[p. 555]{novagapozzetta2020} and using $\mu_{V_y}(\R^3)<\infty$, one finds
\begin{equation}
    \lim_{t\to\infty} A_z(t) = \frac{\W(V_y)}{4} + \frac12 \int_{\R^3} \frac{\langle \eta_{V_y}(x),x-z\rangle}{|x-z|^2}\dd\sigma_{V_y}(x).
\end{equation}
Using \cref{eq:mon-form}, one concludes
\begin{equation}\label{eq:conseq-of-mon-form}
    \theta^2(\mu_{V_y},z)\pi + \int_{\R^3} \Bigl| \frac{\vec{H}_y}{2} + \frac{(x-z)^{\bot}}{|x-z|^2} \Bigr|^2\dd\mu_{V_y} = \frac{\W(V_y)}{4} + \frac12 \int_{\R^3} \frac{\langle \eta_{V_y},x-z\rangle}{|x-z|^2}\dd\sigma_{V_y}.
\end{equation}
We now make a specific choice for $V_y$. There exist $R_y>0$, $x_y\in\R\times\{0\}\times\{0\}$ and a smooth $2$-dimensional submanifold $M_y\subseteq\S^2_{R_y,x_y}\vcentcolon=x_y+R_y\S^2\subseteq\R^3$ such that $V_y\vcentcolon=v(M_y)$ satisfies $\sigma_{V_y}=\Haus^1\measurerestr C_y$, $\eta_{V_y}=\eta_y$. Using \cref{eq:conseq-of-mon-form,eq:def-cly-p-tau}, one finds
\begin{align}
    C_{\mathrm{LY}}&(p,\tau) = 8\pi + \inf_{z\in\R^3}\sum_{y=0,1} \Bigl( \W(V_y) - 4\theta^2(\mu_{V_y},z)\pi - 4\int_{\R^3} \Bigl| \frac{\vec{H}_y}{2} + \frac{(x-z)^{\bot}}{|x-z|^2} \Bigr|^2\dd\mu_{V_y} \Bigr)\\
    &=8\pi+\W(V_0)+\W(V_1) - 4\sup_{z\in\R^3} \sum_{y=0,1}\Bigl( \theta^2(\mu_{V_y},z)\pi + \int_{\R^3} \Bigl| \frac{\vec{H}_y}{2} + \frac{(x-z)^{\bot}}{|x-z|^2} \Bigr|^2\dd\mu_{V_y} \Bigr).\label{eq:rep-of-cly}
\end{align}
\begin{lemma}
    One has
    \begin{equation}
        \W(V_0)+\W(V_1)= T(\tau).
    \end{equation}
\end{lemma}
\begin{proof}
    First consider any $\tau\in\S^1$ with $\tau^{(1)}\neq 0$ and the spherical cut-out $M\subseteq\S^2$ as above with outer co-normal induced by rotating $\tau$ as in the construction above. 

    \textbf{Case 1:} $\tau^{(1)}<0$. Choosing $\theta\in (0,\pi)$ such that
    \begin{equation}
        \begin{pmatrix}
            -\sin(\theta)\\
            \cos(\theta)
        \end{pmatrix} = \tau,
    \end{equation}
    we can parametrize $M$ by the surface of revolution associated to $\gamma(x)=(\cos(x),\sin(x))$ with $x\in[0,\theta)$. Since the mean curvature of $\S^2$ has length $1$ everywhere,
    \begin{equation}\label{eq:app-lem1-1}
        \W(M) = \Haus^2(M) = 2\pi\int_0^{\theta} \gamma^{(2)}(x)|\gamma'(x)|\dd x = 2\pi(1-\cos(\theta)) = 2\pi(1-\tau^{(2)}).
    \end{equation}
    \textbf{Case 2:} $\tau^{(1)}>0$. With similar arguments, one also obtains \cref{eq:app-lem1-1}.

    Using that the Willmore energy is invariant with respect to scaling and translations, that $-\tau_0$ induces the outer co-normal of $V_0$ and that $\tau_1$ induces the outer co-normal of $V_1$, \cref{eq:app-lem1-1} yields
    \begin{align}
        \W(V_0) + \W(V_1) = 2\pi(1+1-\tau_1^{(2)}+\tau_0^{(2)}) = 4\pi -2\pi \tau^{(2)}_y\Big|_{y=0}^1 = T(\tau).\qquad\qedhere
    \end{align}
\end{proof}
\begin{lemma}
    It holds that
    \begin{equation}
        \sup_{z\in\R^3} \sum_{y=0,1}\Bigl( \theta^2(\mu_{V_y},z)\pi + \int_{\R^3} \Bigl| \frac{\vec{H}_y(x)}{2} + \frac{(x-z)^{\bot}}{|x-z|^2} \Bigr|^2\dd\mu_{V_y}(x) \Bigr) \leq 2\pi.
    \end{equation}
\end{lemma}
\begin{proof}
    To this end, simply note that \cref{eq:conseq-of-mon-form} applied to $\S^2_{R_y,x_y}$, a manifold without boundary, yields for $y=0,1$
    \begin{align}
        &\theta^2(\mu_{V_y},z)\pi+\int_{\R^3} \Bigl| \frac{\vec{H}_y(x)}{2} + \frac{(x-z)^{\bot}}{|x-z|^2} \Bigr|^2\dd\mu_{V_y}(x) \\
        &\leq \theta^2(\Haus^2\measurerestr\S^2_{R_y,x_y},z)\pi+ \int_{\S^2_{R_y,x_y}} \Bigl| \frac{\vec{H}_{\S^2_{R_y,x_y}}(x)}{2} + \frac{(x-z)^{\bot}}{|x-z|^2} \Bigr|^2\dd\Haus^2(x) = \frac{\W(\S^2_{R_y,x_y})}{4} = \pi \label{eq:app-lem2-3},
    \end{align}
    using that the support of $V_y$ is contained in the sphere $\S^2_{R_y,x_y}$.
\end{proof}

\begin{proof}[Proof of \cref{prop:battle-of-thresholds}.]
    The claim is concluded by combining \cref{eq:rep-of-cly} with both Lemmata above.
\end{proof}


\section{Explicit computation of $C_{\mathrm{LY}}^{\mathrm{rot}}(p,\tau)$ and comparison to \cite{eichmann2024}}
\label{app:est-cly-rot}

Returning to \cref{eq:def-cly-p-tau}, for $z=h e_1$ where $h\in\R$ and $e_1=(1,0,0)^t$, by direct computation, 
\begin{equation}\label{eq:expl-comp-integral}
    \int_{\partial C} \frac{\langle\eta_{\tau}(x),b_p(x)-z\rangle}{|b_p(x)-z|^2}\dd\mu_{b_p^*\langle\cdot,\cdot\rangle}(x) = 2\pi \Bigl[ p_y^{(2)}\frac{\tau_y^{(1)}(p_y^{(1)}-h)+\tau_y^{(2)}p_y^{(2)}}{(p_y^{(1)}-h)^2+(p_y^{(2)})^2} \Big|_{y=0}^1 \Bigr],
\end{equation}
and plugging this into \cref{eq:def-cly-p-tau} yields the explicit formula \cref{eq:cly-rot-explicit}. Particularly, if one considers the case of horizontal boundary data, i.e. $\tau_0=\tau_1=(1,0)^t$, $p_0=(-1,\alpha_-)^t$ and $p_1=(1,\alpha_+)^t$ for $\alpha_-,\alpha_+>0$, one finds
\begin{align}
    C_{\mathrm{LY}}^{\mathrm{rot}}(p,\tau) &= 4\pi\Bigl[2-\sup_{h\in\R}\Bigl(\frac{\alpha_+(1-h)}{(1-h)^2+(\alpha_+)^2} +\frac{\alpha_-(1+h)}{(1+h)^2+(\alpha_-)^2} \Bigr) \Bigr]\\
    &= 8\pi+4\pi\inf_{h\in\R}\Bigl(\frac{\alpha_+(h-1)}{(1-h)^2+(\alpha_+)^2} -\frac{\alpha_-(1+h)}{(1+h)^2+(\alpha_-)^2} \Bigr)
\end{align}
also cf. \cite[Theorem 1.2]{eichmann2024}. Next, adapting our methodology from the previous section, using Simon's monotonicity formula, we show that the energy thresholds in \cref{eq:en-thresh} and \cite[Theorem 1.1]{eichmann2024} are equivalent for all possible boundary conditions.

To this end, we rewrite the main players in defining the threshold in \cite[Theorem~1.1]{eichmann2024} with this article's notation. If $p_0,p_1\in\H^2$ and $\tau_0,\tau_1\in\S^1$, for $h\in\R$, denote by $c_y^h\colon[0,1)\to\H^2$ for $y=0,1$ a parametrization of a segment of the profile curve of a sphere, a disk or of a Möbius transformed inverted catenoid such that, for $y=0,1$, 
\begin{equation}\label{eq:cyh-bdry}
    c_y^h(0)=p_y,\quad \lim_{x\nearrow 1}c_y^h(x)=(h,0)^t\quad \text{and}\quad\frac{\partial_xc_y^h(0)}{|\partial_xc_y^h(0)|} = (-1)^y \tau_y.
\end{equation}
We then have the following
\begin{proposition}\label{prop:sl-gleich-eic}
    Writing $\Gamma_y=\{y\}\times\S^1$, one finds that
    \begin{equation}\label{eq:interp-of-bdry-integral}
        \W(f_{c_y^h}) = 4\pi - 2\int_{\Gamma_y} \frac{\langle \eta_{\tau}(x),b_p(x)-h e_1\rangle}{|b_p(x)-h e_1|^2}\dd \mu_{b_p^{*}\langle\cdot,\cdot\rangle}(x)
    \end{equation}
    for $y=0,1$, using \cref{eq:def-bp-etatau}, where $f_{c_y^h}$ is the immersion one obtains by rotating $c_y^h$.
\end{proposition}
Especially, summing over $y=0,1$ in \cref{eq:interp-of-bdry-integral}, one finds
\begin{equation}
    \W(f_{c_0^h})+\W(f_{c_1^h}) = 8\pi - 2\int_{\partial C} \frac{\langle \eta_{\tau}(x),b_p(x)-h e_1\rangle}{|b_p(x)-h e_1|^2}\dd \mu_{b_p^{*}\langle\cdot,\cdot\rangle}(x).
\end{equation}
That is, recalling \cref{eq:def-cly-p-tau}, the energy thresholds of \cref{thm:main} and \cite[Theorem 1.1]{eichmann2024} coincide.
\begin{proof}[Proof of \cref{prop:sl-gleich-eic}.]
    Suppose first of all that $c^h_y$ parametrizes a segment of the profile curve of a Möbius transformed inverted catenoid. Moreover, denote by $V_{c^h_y}$ the varifold one obtains by rotating $c^h_y$ and by $V_{c^h_y}^{\mathrm{cl}}$ the associated inverted catenoid without boundary whose support contains the support of $V_{c_y^h}$. Using \cref{eq:cyh-bdry}, one finds $\theta^2(\mu_{V_{c^h_y}^{\mathrm{cl}}},he_1)\pi=2\pi$, i.e. the closed inverted catenoid has density $2$ at the point $he_1$ where it intersects the rotation axis. Moreover, as $\W(V_{c^h_y}^{\mathrm{cl}})=8\pi$, proceeding as in \cref{eq:conseq-of-mon-form} yields
    \begin{equation}
        2\pi + \int_{\R^3} \Bigl| \frac{\vec{H}_{V_{c^h_y}^{\mathrm{cl}}}(x)}{2} + \frac{(x-he_1)^{\bot}}{|x-he_1|^2} \Bigr|^2\dd\mu_{V_{c^h_y}^{\mathrm{cl}}}(x) = 2\pi.
    \end{equation}
    Particularly, 
    \begin{equation}
        0\leq \int_{\R^3} \Bigl| \frac{\vec{H}_{V_{c^h_y}}(x)}{2} + \frac{(x-he_1)^{\bot}}{|x-he_1|^2} \Bigr|^2\dd\mu_{V_{c^h_y}}(x) \leq \int_{\R^3} \Bigl| \frac{\vec{H}_{V_{c^h_y}^{\mathrm{cl}}}(x)}{2} + \frac{(x-he_1)^{\bot}}{|x-he_1|^2} \Bigr|^2\dd\mu_{V_{c^h_y}^{\mathrm{cl}}}(x) = 0.
    \end{equation}
    By \cref{eq:cyh-bdry}, we have $\theta^2(\mu_{V_{c^h_y}},he_1)\pi=\pi$, i.e. the varifold obtained by rotating $c^h_y$ only has density equal to $1$ at the point where it intersects the rotation axis. Therefore, \cref{eq:conseq-of-mon-form} applied to $V_{c_y^h}$ yields \cref{eq:interp-of-bdry-integral}, using \cref{eq:def-bp-etatau}. 
    
    With the very same argument, one proves the claim also in the case where $c^h_y$ parametrizes a segment of the profile curve of a sphere.

    Finally, if $c^h_y$ parametrizes the profile curve of a disk, then $\W(f_{c_0^h})=0$. Moreover, using \cref{eq:cyh-bdry}, $\tau_y=(-1)^y(0,-1)^t$ and $p_y^{(1)}=h$. Thus, as in \cref{eq:expl-comp-integral}, using \cref{eq:def-bp-etatau},
    \begin{align}
        \int_{\Gamma_y} \frac{\langle \eta_{\tau}(x),b_p(x)-h e_1\rangle}{|b_p(x)-h e_1|^2}\dd \mu_{b_p^{*}\langle\cdot,\cdot\rangle}(x) &= 2\pi (-1)^{y+1} \Bigl( p_y^{(2)}\frac{\tau_y^{(1)}(p_y^{(1)}-h)+\tau_y^{(2)}p_y^{(2)}}{(p_y^{(1)}-h)^2+(p_y^{(2)})^2} \Bigr)\\
        &=2\pi (-1)^{y+1} \tau_y^{(2)} = 2\pi
    \end{align}
    and \cref{eq:interp-of-bdry-integral} also follows in this case.
\end{proof}

\section*{Acknowledgment}
The author would like to thank Anna Dall’Acqua for helpful discussions and Sascha Eichmann for his comments on comparing \cref{thm:main} and \cite[Theorem 1.1]{eichmann2024}.


\begin{thebibliography}{DMSS20}

    \bibitem[All72]{allard1972}
    William~K. Allard.
    \newblock On the first variation of a varifold.
    \newblock {\em Ann. of Math. (2)}, 95:417--491, 1972.
    
    \bibitem[BG86]{bryantgriffiths1986}
    Robert Bryant and Phillip Griffiths.
    \newblock Reduction for constrained variational problems and {$\int{1\over
      2}k^2\,ds$}.
    \newblock {\em Amer. J. Math.}, 108(3):525--570, 1986.
    
    \bibitem[Bry84]{bryant1984}
    Robert Bryant.
    \newblock A duality theorem for {W}illmore surfaces.
    \newblock {\em J. Differential Geom.}, 20(1):23--53, 1984.
    
    \bibitem[Can70]{canham1970}
    Peter Canham.
    \newblock The minimum energy of bending as a possible explanation of the
      biconcave shape of the human red blood cell.
    \newblock {\em J. Theor. Biol.}, 26(1):61--81, 1970.
    
    \bibitem[DDG08]{dallacquadeckelnickgrunau2008}
    Anna Dall'Acqua, Klaus Deckelnick, and Hans-Christoph Grunau.
    \newblock Classical solutions to the {D}irichlet problem for {W}illmore
      surfaces of revolution.
    \newblock {\em Adv. Calc. Var.}, 1(4):379--397, 2008.
    
    \bibitem[DFGS11]{dallacquafroehlichgrunauschieweck2011}
    Anna Dall'Acqua, Steffen Fr\"{o}hlich, Hans-Christoph Grunau, and Friedhelm
      Schieweck.
    \newblock Symmetric {W}illmore surfaces of revolution satisfying arbitrary
      {D}irichlet boundary data.
    \newblock {\em Adv. Calc. Var.}, 4(1):1--81, 2011.
    
    \bibitem[DMSS20]{dallacquamullerschatzlespener2020}
    Anna Dall'Acqua, Marius M\"{u}ller, Rainer Sch\"{a}tzle, and Adrian Spener.
    \newblock The {W}illmore flow of tori of revolution.
    \newblock {\em arXiv:2005.13500 [math.AP]}, 2020.
    
    \bibitem[DS17]{dallacquaspener2017}
    Anna Dall'Acqua and Adrian Spener.
    \newblock The elastic flow of curves in the hyperbolic plane.
    \newblock {\em arXiv:1710.09600 [math.AP]}, 2017.
    
    \bibitem[DS18]{dallacquaspener2018}
    Anna Dall'Acqua and Adrian Spener.
    \newblock Circular solutions to the elastic flow in hyperbolic space.
    \newblock {\em Proceedings of the conference Analysis on Shapes of Solutions to
      Partial Differential Equations}, 2082:109--124, 2018.
    
    \bibitem[EG19]{eichmanngrunau2019}
    Sascha Eichmann and Hans-Christoph Grunau.
    \newblock Existence for {W}illmore surfaces of revolution satisfying
      non-symmetric {D}irichlet boundary conditions.
    \newblock {\em Adv. Calc. Var.}, 12(4):333--361, 2019.
    
    \bibitem[Eic16]{eichmann2016}
    Sascha Eichmann.
    \newblock Nonuniqueness for {W}illmore surfaces of revolution satisfying
      {D}irichlet boundary data.
    \newblock {\em J. Geom. Anal.}, 26(4):2563--2590, 2016.
    
    \bibitem[Eic24]{eichmann2024}
    Sascha Eichmann.
    \newblock {I}mproved long time existence for the {W}illmore flow of surfaces of
      revolution with {D}irichlet data.
    \newblock {\em arXiv:2402.05580 [math.AP]}, 2024.
    
    \bibitem[ES24]{eichmannschaetzle2024}
    Sascha Eichmann and Reiner Sch\"{a}tzle.
    \newblock The rotational symmetric {W}illmore boundary problem.
    \newblock {\em Preprint:
      \url{https://www.math.uni-tuebingen.de/user/schaetz/publi/eichmann-schaetzle-24.pdf}
      date of download 12.02.2024}, 2024.
    
    \bibitem[Haw68]{hawking1968}
    Stephen Hawking.
    \newblock Gravitational radiation in an expanding universe.
    \newblock {\em J. Math. Phys.}, 9(4):598--604, 1968.
    
    \bibitem[Hel73]{helfrich1973}
    Wolfgang Helfrich.
    \newblock Elastic properties of lipid bilayers: theory and possible
      experiments.
    \newblock {\em Zeitschrift f{\"u}r Naturforschung C}, 28(11-12):693--703, 1973.
    
    \bibitem[KS01]{kuwertschaetzle2001}
    Ernst Kuwert and Reiner Sch\"{a}tzle.
    \newblock The {W}illmore flow with small initial energy.
    \newblock {\em J. Differential Geom.}, 57(3):409--441, 2001.
    
    \bibitem[KS02]{kuwertschaetzle2002}
    Ernst Kuwert and Reiner Sch\"{a}tzle.
    \newblock Gradient flow for the {W}illmore functional.
    \newblock {\em Comm. Anal. Geom.}, 10(2):307--339, 2002.
    
    \bibitem[KS04]{kuwertschaetzle2004}
    Ernst Kuwert and Reiner Sch\"{a}tzle.
    \newblock Removability of point singularities of {W}illmore surfaces.
    \newblock {\em Ann. of Math. (2)}, 160(1):315--357, 2004.
    
    \bibitem[Lee18]{lee2018}
    John~M. Lee.
    \newblock {\em Introduction to {R}iemannian manifolds}, volume 176 of {\em
      Graduate Texts in Mathematics}.
    \newblock Springer, Cham, second edition, 2018.
    
    \bibitem[LS84]{langersinger1984}
    Joel Langer and David Singer.
    \newblock The total squared curvature of closed curves.
    \newblock {\em J. Differential Geom.}, 20(1):1--22, 1984.
    
    \bibitem[LY82]{liyau1982}
    Peter Li and Shing~Tung Yau.
    \newblock A new conformal invariant and its applications to the {W}illmore
      conjecture and the first eigenvalue of compact surfaces.
    \newblock {\em Invent. Math.}, 69(2):269--291, 1982.
    
    \bibitem[Min12]{mingione2012}
    Giuseppe Mingione.
    \newblock {\em Topics in modern regularity theory}, volume~13.
    \newblock Springer Science \& Business Media, 2012.
    
    \bibitem[MS64]{meyersserrin1964}
    Norman Meyers and James Serrin.
    \newblock {$H=W$}.
    \newblock {\em Proc. Nat. Acad. Sci. U.S.A.}, 51:1055--1056, 1964.
    
    \bibitem[MS20]{muellerspener2020}
    Marius M\"{u}ller and Adrian Spener.
    \newblock On the convergence of the elastic flow in the hyperbolic plane.
    \newblock {\em Geom. Flows}, 5(1):40--77, 2020.
    
    \bibitem[NP20]{novagapozzetta2020}
    Matteo Novaga and Marco Pozzetta.
    \newblock Connected surfaces with boundary minimizing the {W}illmore energy.
    \newblock {\em Math. Eng.}, 2(3):527--556, 2020.
    
    \bibitem[RS23]{ruppscharrer2023}
    Fabian Rupp and Christian Scharrer.
    \newblock Li-{Y}au inequalities for the {H}elfrich functional and applications.
    \newblock {\em Calc. Var. Partial Differential Equations}, 62(2):Paper No. 45,
      43, 2023.
    
    \bibitem[Sch10]{schaetzle2010}
    Reiner Sch\"{a}tzle.
    \newblock The {W}illmore boundary problem.
    \newblock {\em Calc. Var. Partial Differential Equations}, 37(3-4):275--302,
      2010.
    
    \bibitem[Sch22]{scharrer2022}
    Christian Scharrer.
    \newblock Some geometric inequalities for varifolds on {R}iemannian manifolds
      based on monotonicity identities.
    \newblock {\em Ann. Global Anal. Geom.}, 61(4):691--719, 2022.
    
    \bibitem[Sch23]{schlierf2023}
    Manuel Schlierf.
    \newblock {S}ingularities of the hyperbolic elastic flow: {C}onvergence,
      quantization and blow-ups.
    \newblock {\em arXiv:2311.05978 [math.AP]}, 2023.
    
    \bibitem[Sch24]{schlierf2024}
    Manuel Schlierf.
    \newblock On the convergence of the {W}illmore flow with {D}irichlet boundary
      conditions.
    \newblock {\em Nonlinear Anal.}, 241:113475, 2024.
    
    \bibitem[Sim93]{simon1993}
    Leon Simon.
    \newblock Existence of surfaces minimizing the {W}illmore functional.
    \newblock {\em Comm. Anal. Geom.}, 1(2):281--326, 1993.
    
    \end{thebibliography}

\end{document}